\documentclass[a4paper,10pt]{amsart}
\usepackage{enumerate, amsmath, amsfonts, amssymb, amsthm, mathtools, thmtools, wasysym, graphics, graphicx, xcolor, frcursive,xparse,comment,ytableau,bbm,arydshln,epic,comment}
\usepackage{ mathrsfs }
\usepackage{multirow}
\usepackage{rotating}
\usepackage{appendix}
\usepackage[all]{xy}
\usepackage{hyperref}
\usepackage{csquotes}

\usepackage{url, hypcap}
\hypersetup{colorlinks=true, citecolor=darkblue, linkcolor=darkblue}

\usepackage{tikz}
\usetikzlibrary{calc,through,backgrounds,shapes,matrix}

\definecolor{darkblue}{rgb}{0.0,0,0.7} 
\newcommand{\darkblue}{\color{darkblue}} 
\definecolor{darkred}{rgb}{0.7,0,0} 
\definecolor{lightgrey}{rgb}{0.7,0.7,0.7} 

\usepackage{graphicx}

\def\ZZ{\mathbb{Z}}
\def\RR{\mathbb{R}}
\def\CC{\mathbb{C}}
\def\NN{\mathbb{N}}

\newcommand{\Hilb}{\mathrm{Hilb}}

\newcommand{\spn}{{\mathrm{span}}}

\newcommand{\fix}{{\mathrm{fix}}}
\newcommand{\g}{{G}}
\newcommand{\n}{{N}}
\newcommand{\CN}{{\mathcal{C}_N}}

\newcommand{\UN}{{U^N}}
\newcommand{\UNs}{{U^N_\sigma}}
\newcommand{\UNd}{{(U^N)^*}}
\newcommand{\UNsd}{{(U^N_\sigma)^*}}
\newcommand{\UNsde}{{(U^N_\sigma)^*_e}}
\newcommand{\UNM}{{U_M^N}}
\newcommand{\UNMd}{{(U_M^N)^*}}
\newcommand{\UNMe}{{(U_M^N)_e}}
\newcommand{\UNMde}{{(U_M^N)^*_e}}

\newcommand{\UGstd}{{(\tilde{U}^G_\sigma)^*}}
\newcommand{\CG}{{\mathcal{C}_G}}
\newcommand{\IG}{{I_G^+}}
\newcommand{\IW}{{I_W^+}}

\newcommand{\UG}{{U^G}}
\newcommand{\UGs}{{U^G_\sigma}}
\newcommand{\UGd}{{(U^G)^*}}
\newcommand{\UGsd}{{(U^G_\sigma)^*}}
\newcommand{\UGM}{{U_M^G}}
\newcommand{\UGMd}{{(U_M^G)^*}}
\newcommand{\CH}{{\mathcal{C}_H}}

\newcommand{\Vs}{{V^\sigma}}
\newcommand{\Vsd}{{(V^\sigma)^*}}
\newcommand{\Vd}{{V^*}}

\newcommand{\ENM}{{\mathcal{E}_M^N}}
\newcommand{\ENs}{{\mathcal{E}^N_\sigma}}

\newcommand{\h}{{H}}
\newcommand{\VN}{{E}}
\newcommand{\VNd}{{E^*}}
\newcommand{\VNdd}{{E^*_d}}
\newcommand{\VNs}{{E^\sigma}}
\newcommand{\VNsd}{{(E^\sigma)^*}}
\newcommand{\DN}{{\mathcal{D}_N}}
\newcommand{\fixed}{{\mathrm{fix}}}

\renewcommand{\mod}{{\text{ mod }}}
\renewcommand{\P}{\mathcal{P}}

\newtheorem{theorem}{Theorem}[section]
\newtheorem{proposition}[theorem]{Proposition}

\newtheorem{corollary}[theorem]{Corollary}
\newtheorem{lemma}[theorem]{Lemma}

\theoremstyle{definition}
\newtheorem{definition}[theorem]{Definition}
\newtheorem{example}[theorem]{Example}

\newtheorem{remark}[theorem]{Remark}

\usepackage[nameinlink,capitalize]{cleveref}
\usepackage[all]{xy}
\usepackage[T1]{fontenc}
\crefformat{footnote}{#2\footnotemark[#1]#3}
\crefformat{conjecture}{Conjecture~#2#1#3}
\crefformat{fact}{Fact~#2#1#3}
\AtBeginEnvironment{appendices}{\crefalias{section}{appendix}}

\DeclareFontFamily{U}{rcjhbltx}{}
\DeclareFontShape{U}{rcjhbltx}{m}{n}{<->rcjhbltx}{}
\DeclareSymbolFont{hebrewletters}{U}{rcjhbltx}{m}{n}
\DeclareMathSymbol{\lamed}{\mathord}{hebrewletters}{108}
\usepackage[colorinlistoftodos]{todonotes}

\newcommand{\nathan}[1]{\todo[size=\tiny,color=green!30]{#1 \\ \hfill --- N.}}

\newcommand{\carlos}[1]{\todo[size=\tiny,color=orange!30]{#1 \\ \hfill --- C.}}

\newcommand{\defn}[1]{\emph{\darkblue #1}}
\usetikzlibrary{math}
\usepackage{xifthen}
\usepackage{xstring}
\newcommand\IfStringInList[2]{\IfSubStr{,#2,}{,#1,}}

\title[Normal Reflection Subgroups]
  {Normal Reflection Subgroups \\ of Complex Reflection Groups}

\author[C.~Arreche]{Carlos E. Arreche}
\address[C.~Arreche]{The University of Texas at Dallas}
\email{arreche@utdallas.edu}
\thanks{C.~Arreche was partially supported by NSF grant CCF-1815108.}

\author[N.~Williams]{Nathan F. Williams}
\address[N.~Williams]{The University of Texas at Dallas}
\email{nathan.f.williams@gmail.com}
\thanks{N.~Williams was partially supported by Simons Foundation award number 585380.}

\date{\today}

\begin{document}

\maketitle

\begin{abstract}
We study normal reflection subgroups of complex reflection groups. Our approach leads to a refinement of a theorem of Orlik and Solomon to the effect that the generating function for fixed-space dimension over a reflection group is a product of linear factors involving generalized exponents. Our refinement gives a uniform proof and generalization of a recent theorem of the second author.
\end{abstract}

\section{Introduction}
\label{sec:intro}

\subsection{Lie Groups}
Hopf proved that the cohomology of a real connected compact Lie group $\mathcal{G}$ is a free exterior algebra on $r=\mathrm{rank}(\mathcal{G})$ generators of odd degree~\cite{hopf1964topologie}.  Its Poincar\'e series is therefore given by \begin{equation}\label{eq:hopf}\mathrm{Hilb}(H^*(\mathcal{G});q) = \prod_{i=1}^r (1+q^{2e_i+1}).\end{equation}
Chevalley presented these $e_i$ for the exceptional simple Lie algebras in his 1950 address at the International Congress of Mathematicians~\cite{chevalley1950betti}, and Coxeter recognized them from previous work with real reflection groups~\cite{coxeter1951product}.
This observation has led to deep relationships between the cohomology of $\mathcal{G}$ and the invariant theory of the corresponding Weyl group $W=N_\mathcal{G}(T)/T$, where $T$ is a maximal torus in $\mathcal{G}$~\cite{reeder1995cohomology,reiner2019invariant}---notably,
\[H^*(\mathcal{G}) \simeq \left(H^*(\mathcal{G}/T) \times H^*(T)\right)^W \simeq \left(S(V^*)/\IW \otimes \bigwedge V^*\right)^W\!\!\!\!\!,\] where $V=\mathrm{Lie}(T)$ is the reflection representation of $W$, $S(V^*)$ is the algebra of polynomial functions on $V$, and $\IW$ is the ideal generated by the $W$-invariant polynomials in $S(V^*)$ with no constant term.  For more details, we refer the reader to the wonderful survey~\cite{barcelo1994combinatorial}.

\subsection{Complex Reflection Groups}
It turns out that the $e_i$ in \Cref{eq:hopf} can be computed from the generating function for the dimension of the fixed space $\fixed(w):=\dim(\ker(1-w))$ for $w\in W$, via the remarkable formula:
\begin{equation}\sum_{w \in W} q^{\fixed(w)} = \prod_{i=1}^r (q+e_i).\label{eq:shephard_todd}\end{equation}

Shephard and Todd verified case-by-case that the same sum still factors when $W$ is replaced by a finite complex reflection group $G\subset\mathrm{GL}(V)$ acting by reflections on a complex vector space $V$ of dimension $r$~\cite[Theorem 5.3]{shephard1954finite}. The $e_i$ are now determined by the degrees $d_i$ of the fundamental invariants of $G$ on $V$ as $e_i=d_i-1$.  A case-free proof of this result was given by Solomon in~\cite{solomon1963invariants}, mirroring Hopf's result:  $\left(S(V^*) \otimes \bigwedge V^*\right)^G$ is a free exterior algebra over the ring $S(V^*)^G$ of $G$-invariant polynomials, which gives a factorization of the Poincar\'e series of the $G$-invariant differential forms

\begin{equation}
\label{eq:solomon}
\mathrm{Hilb}\left(\left(S(V^*) \otimes \bigwedge V^* \right)^G;q,u\right) = \prod_{i=1}^r \frac{1+u q^{e_i}}{1-q^{d_i}}.
\end{equation}
Computing the trace of the projection $\frac{1}{|G|}\sum_{g\in G} g$ to the subspace of $G$-invariants on $S(V^*) \otimes \bigwedge V^*$, specializing to $u=q(1-x)-1$, and taking the limit as $x \to 1$ gives the Shephard-Todd result in~\Cref{eq:shephard_todd}.

\subsection{Galois twists and cohomology}
More generally, the \defn{fake degree} of an $m$-dimensional simple $G$-module $M$ is the polynomial encoding the degrees in which $M$ occurs in the coinvariant algebra $S(V^*)/\IG\simeq \CG$: \begin{equation}f_M(q)=\sum_{i} ((\CG)_i,M)q^i =\sum_{i=1}^m q^{e_i(M)}.\label{eq:fake}\end{equation}  The fake degree of a reducible $G$-module is defined to be the sum of the fake degrees of its simple direct summands. The integers $e_i(M)$ in \Cref{eq:fake} are called the \defn{$M$-exponents of $G$}.

Letting $\zeta_G$ denote a primitive $|G|$-th root of unity, for $\sigma\in\mathrm{Gal}(\mathbb{Q}(\zeta_G)/\mathbb{Q})$ the \defn{Galois twist} $V^\sigma$ is the representation of $G$ obtained by applying $\sigma$ to its matrix entries. In~\cite{orlik1980unitary}, Orlik and Solomon gave a beautiful generalization of \Cref{eq:solomon,eq:shephard_todd} that takes into account these Galois twists (see~\Cref{sec:o-s_thm}). 

\begin{theorem}[{\cite[Thm.~3.3]{orlik1980unitary}}]
\label{thm:orlik_solomon}
Let $G\subset\mathrm{GL}(V)$ be a complex reflection group of rank $r$ and let $\sigma \in \mathrm{Gal}(\mathbb{Q}(\zeta_G)/\mathbb{Q})$. Then 
\[\sum_{g \in G} \left(\prod_{\lambda_i(g) \neq 1} \frac{1-\lambda_i(g)^\sigma}{1-\lambda_i(g)}\right) q^{\fix_{V} (g)} = \prod_{i=1}^r \left(q+e_i(V^\sigma) \right),\] where the $\lambda_i(g)$ are the eigenvalues of $g \in G$ acting on $V$.
\end{theorem}

When $\sigma:\zeta_G\mapsto \overline{\zeta_G}$ is complex conjugation, Orlik and Solomon~\cite[Thm.~4.8]{orlik1980unitary} further connected their~\Cref{thm:orlik_solomon} to the cohomology of the complement of the corresponding hyperplane arrangement---in this case $V^\sigma\simeq V^*$ as a $G$-representation, and the co-exponents $e_i(V^*)$ are the degrees of the generators of the cohomology ring of the complement of the hyperplane arrangement.

\subsection{Normal Reflection Subgroups of Complex Reflection Groups}
Let $G \subset \mathrm{GL}(V)$ be a complex reflection group. We say that $N \trianglelefteq G$ is a \defn{normal reflection subgroup} of $G$ if it is a normal subgroup of $G$ that is generated by reflections.  The main theorem of this paper, \Cref{thm:main_theorem}, gives a new refinement of~\Cref{thm:orlik_solomon} to accommodate a normal reflection subgroup.  The following result is a special case of~\cite{bessis2002quotients}, where they consider the more general notion of ``bon sous-groupe distingu\'e'' in lieu of our normal reflection subgroup $N$ of $G$. 

\begin{theorem}
\label{thm:reflection_action}
Let $G\subset \mathrm{GL}(V)$ be a complex reflection group and let $N\trianglelefteq G$ be a normal reflection subgroup. Then $G/N=H$ acts as a reflection group on the vector space $V/N=\VN$.
\end{theorem}

The ``bon sous-groupes distingu\'es'' of~\cite{bessis2002quotients} are precisely those normal subgroups for which the associated quotient group is a reflection group acting on the tangent space at $0$ of $V/N$, which is a strictly weaker condition than being a normal reflection subgroup. Our proof of \Cref{thm:reflection_action} in~\Cref{sec:quotients} follows the ideas of~\cite{bessis2002quotients}, but specialized to our more restricted setting where the normal subgroup under consideration is actually a normal reflection subgroup. In this more restricted setting, we are able to prove the new results \Cref{thm:numbers,thm:main_theorem} stated below.

The technical definition of the $G$-module $\UNs$ that mediates the statement of the following result is given in \Cref{def:OS_space}. Since we are dealing with multiple reflection groups acting on multiple spaces, we will begin labeling exponents and degrees by their corresponding groups.

\begin{theorem}
\label{thm:numbers}
Let $G\subset \mathrm{GL}(V)$ be a complex reflection group and let $N\trianglelefteq G$ be a normal reflection subgroup. Let $H=G/N$ and $\VN=V/N$. Then for a suitable choice of indexing we have
\begin{align*}
e_i^N(\Vs) {+} e_i^G(\UNs) & = e_i^G(\Vs)\\
d_i^N \cdot e_i^H(\VNs) & = e_i^G(\VNs)\\
d_i^N\cdot d_i^H & =d_i^G.
\end{align*}
\end{theorem}

In the special case $\sigma=1$, it is well-known that $\UNs\simeq \VN$ as $G$-modules (see \Cref{def:OS_space} and \Cref{EU}), so that \Cref{thm:numbers} coincides with \cite[Theorem~1.3]{arreche2020normal} in this case. As we explain in \Cref{rem:untwisted_numbers}, in this special case where $\sigma=1$ the equalities in \Cref{thm:numbers} are compatible with the relations $d_i=e_i+1$ between classical exponents and degrees for the three reflection groups involved.

An essential tool in our proof of \Cref{thm:numbers} is \Cref{prop:coinvariant_decomposition}, which gives a graded $G$-module isomorphism  $\CG\simeq\CH\otimes\CN$ relating the spaces of harmonic polynomials for $N$ and $H$ to that for $G$, which is an interesting and useful result in its own right.

Our~\Cref{thm:main_theorem} below generalizes the Orlik-Solomon formula from~\Cref{thm:orlik_solomon} to take into account the additional combinatorial data arising from a normal reflection subgroup. The technical definition of the $G$-module $\UNs$ is again given in \Cref{def:OS_space}.

\begin{theorem}
\label{thm:main_theorem}
Let $G\subset \mathrm{GL}(V)$ be a complex reflection group of rank $r$ and let $N\trianglelefteq G$ be a normal reflection subgroup. Let $\VN=V/N$ and $\sigma\in\mathrm{Gal}(\mathbb{Q}(\zeta_G)/\mathbb{Q})$. Then for a suitable choice of indexing we have
\[\sum_{g \in G}\left( \prod_{\lambda_i(g) \neq 1} \frac{1-\lambda_i(g)^\sigma}{1-\lambda_i(g)}\right) q^{\fix_V (g)} t^{\fix_{E} (g)} = \prod_{i=1}^r \left(qt+e_i^N(V^\sigma) t + e_i^G(\UNs)\right),\] %
where the $\lambda_i(g)$ are the eigenvalues of $g\in G$ acting on $V$.
\end{theorem}

In view of \Cref{thm:numbers}, specializing to $t=1$ recovers \Cref{thm:orlik_solomon}. Moreover, since $\UNs\simeq\VN$ as $G$-modules when $\sigma=1$ (see again \Cref{def:OS_space} and \Cref{EU}), \Cref{thm:main_theorem} coincides with \cite[Theorem~1.5]{arreche2020normal} in this case, which when similarly specialized to $t=1$ recovers \Cref{eq:shephard_todd}. As explained in \Cref{rem:main_theorem_specializations}, one can also recover \Cref{thm:orlik_solomon} for the reflection group $N$ from \Cref{thm:main_theorem} by applying $\frac{1}{r!}\frac{\partial^r}{\partial t^r}$ on both sides. In the special case $\sigma=1$, one can recover \Cref{eq:shephard_todd} for the reflection group $H$ by specializing \Cref{thm:main_theorem} to $q=1$ and dividing by $|N|$ on both sides, but this same specialization does not seem to be directly related to \Cref{thm:orlik_solomon} for $H$ in general for arbitrary $\sigma\in\mathrm{Gal}(\mathbb{Q}(\zeta_G)/\mathbb{Q})$.

Our proof of \Cref{thm:main_theorem} follows a similar strategy to the one employed in~\cite{orlik1980unitary}: we compute the Poincar\'e series for $(S(V^*)\otimes\bigwedge \UNsd)^G$ in two equivalent and standard ways, and then obtain \Cref{thm:main_theorem} from a well-chosen specialization. However, a delicate technical issue arises in that our specialization does not provide the correct contribution term-by-term in the left-hand side of \Cref{thm:main_theorem}. We overcome this technical difficulty by applying the results of~\cite{bonnafe2006twisted}, where the authors develop a ``twisted invariant theory'' for cosets $Ng$ of a reflection group $N\subset\mathrm{GL}(V)$ for $g\in\mathrm{GL}(V)$ an element of the normalizer of $N$ in $\mathrm{GL}(V)$. Our proof of \Cref{thm:main_theorem} applies the results of \cite{bonnafe2006twisted} to the special situation where the cosets $Ng$ all come from $g\in G$, a reflection group containing $N$ as a normal reflection subgroup, to show that our specialization argument does provide the correct contribution coset-by-coset.



 
In summary, we have applied results and insights from \cite{bessis2002quotients,bonnafe2006twisted} in the development of new results in the invariant theory for complex reflection groups $G$ taking into account the additional combinatorial data arising from normal reflection subgroups $N\trianglelefteq G $ and their corresponding reflection group quotients $H=G/N$. The setting of \cite{bessis2002quotients} considers more general $N$ (their ``bon sous-groups distingu\'es''), whereas the setting of \cite{bonnafe2006twisted} considers more general cosets $Ng$ (for arbitrary $g$ in the normalizer of $N$). The general study of normal reflection subgroups initiated in this paper is both natural, since it lies in the intersection \cite{bessis2002quotients}$\cap$\cite{bonnafe2006twisted}, as well as productive, as evidenced for example by \Cref{thm:numbers,thm:main_theorem}, the graded $G$-module isomorphism $\CG\simeq\CH\otimes\CN$ of \Cref{prop:coinvariant_decomposition}, and the ancillary results in \Cref{sec:quotients} relating the amenability of different modules with respect to the groups $G$, $N$, and $H$.

\subsection{Organization} We recall standard results about complex reflection groups in \Cref{sec:invariant_theory}. In \Cref{sec:quotients}, we introduce normal reflection subgroups and prove ancillary results, relating spaces of harmonic polynomials and amenability with respect to different reflection groups. We prove the main results stated in the introduction, \Cref{thm:numbers,thm:main_theorem}, in \Cref{sec:proofs}. In \Cref{sec:reflexponents} we recall the case-by-case results of \cite{williams2019reflexponents} and discuss how they are obtained in a case-free way by the methods of the present paper. In \Cref{sec:classification} we provide a complete classification of the normal reflection subgroups of the irreducible complex reflection groups. Finally, in \Cref{sec:examples} we give several examples that illustrate our general results. 

 \section{Invariant Theory of Reflection Groups}
 \label{sec:invariant_theory}
Let $V$ be a complex vector space of dimension $r$. A \defn{reflection} is an element of $\mathrm{GL}(V)$ of finite order that fixes some hyperplane pointwise. A \defn{complex reflection group} $G$ is a finite subgroup of $\mathrm{GL}(V)$ that is generated by reflections. A complex reflection group $G$ is called \defn{irreducible} if $V$ is a simple $G$-module; $V$ is then called the \defn{reflection representation} of $G$. A \defn{(normal) reflection subgroup} of $G$ is a (normal) subgroup that is generated by reflections. In what follows, a \defn{$G$-module} will always be a complex representation of $G$.

\subsection{Chevalley-Shephard-Todd's Theorem} Let $S(V^*)$ be the symmetric algebra on the dual vector space $V^*$, and write $S(V^*)^G$ for its $G$-invariant subring.  By a classical theorem of Shephard-Todd~\cite{shephard1954finite} and Chevalley~\cite{chevalley1955invariants}, a finite subgroup $G \subset \mathrm{GL}(V)$ is a complex reflection group if and only if $S(V^*)^G$ is a polynomial ring, and in this case $S(V^*)^G$ is generated by $r$ algebraically independent homogeneous polynomials---the \defn{degrees} $d_1\leq \cdots \leq d_r$ of these polynomials are invariants of $G$.

\begin{theorem}[\cite{shephard1954finite,chevalley1955invariants}]
\label{thm:shephard_todd}
A finite subgroup $G \subset \mathrm{GL}(V)$ is a complex reflection group if and only if there exist $r=\mathrm{dim}(V)$ homogeneous algebraically independent polynomials $\g_1,\ldots,\g_r$ such that $S(V^*)^G=\CC[\g_1,\ldots,\g_r]$. In this case, $|G|=\prod_{i=1}^r d_i$, where $d_i=\mathrm{deg}(\g_i)$.
\end{theorem}

Let $\IG\subset S(V^*)$ denote the ideal generated by homogeneous $G$-invariant polynomials of positive degree.  In~\cite{chevalley1955invariants}, Chevalley proved that, as an ungraded $G$-module, $S(V^*)/\IG$ affords the regular representation of $G$. Since $\IG$ is $G$-stable, we may choose a $G$-stable complement $\CG\subset S(V^*)$, so that $S(V^*)\simeq \IG\oplus \CG$ as graded $G$-modules, and $\CG$ is a graded version of the regular representation of $G$. Chevalley also proved in \cite{chevalley1955invariants} that $S(V^*)\simeq S(V^*)^G\otimes\CG$ as graded $G$-modules. A canonical choice for such a $G$-stable complement $\CG$ is the space of \defn{$G$-harmonic polynomials} \cite[Corollary~9.37]{lehrer2009unitary}, that is, polynomials in $S(V^*)$ that are annihilated by all $G$-invariant polynomial differential operators with no constant term \cite[Definition~9.35]{lehrer2009unitary}. 

\begin{remark}\label{rem:harmonic_stable}The space $\CG$ of $G$-harmonic polynomials is stabilized by the normalizer of $G$ in $\mathrm{GL}(V)$ \cite[Proposition~12.2]{lehrer2009unitary}. This fact will be essential in our treatment of normal reflection subgroups.

Our choice of notation $\mathcal{C}$ for the space of harmonic polynomials, instead of the more common and natural $\mathcal{H}$ used in the literature, is meant to avoid unfortunate phonetic confusion with the quotient group $H=G/N$ that will play a prominent role in the rest of the paper.
\end{remark}

\subsection{Solomon's Theorem}
We recall the following celebrated theorem of Solomon.

\begin{theorem}[{\cite{solomon1963invariants}}]
\label{thm:solomon}  $\left(S(V^*)\otimes \bigwedge V^*\right)^G$ is a free exterior algebra over the ring of $G$-invariant polynomials:
\[\left(S(V^*)\otimes \bigwedge V^*\right)^G \simeq S(V^*)^G \otimes \bigwedge \UGd,\] where $\UGd=\spn_\CC\left\{d\g_1,\ldots,d\g_r\right\}$ and $d\g_i=\sum_{j=1}^r \frac{\partial \g_i}{\partial x_j}  \otimes x_j$ form a free basis for $(S(V^*)\otimes V^*)^G$ over $S(V^*)^G$.
\end{theorem}

\noindent Computing the trace on $S(V^*) \otimes \bigwedge V^*$ of the projection to the $G$-invariants $\frac{1}{|G|}\sum_{g\in G} g$ gives a formula for the Poincar\'e series as a sum over the group.

\begin{corollary}[{\cite{solomon1963invariants}}]
\label{cor:sol1}
\[\Hilb\left(\left(S(V^*)\otimes \bigwedge V^* \right)^G;x,u\right) = \frac{1}{|G|} \sum_{g \in G} \frac{\mathrm{det}(1+ug |_V)}{\det(1-x g|_V)} = \prod_{i=1}^r\frac{1+x^{e_i^G(V)}u}{1-x_i^{d_i^G}}.\]
\end{corollary}

\noindent Specializing \Cref{cor:sol1} to $u=q(1-x)-1$, and taking the limit as $x \to 1$ gives the Shephard-Todd formula from \Cref{eq:shephard_todd}.

\subsection{Orlik-Solomon's Theorem}
\label{sec:o-s_thm}

The reflection representation $V$ of $G\subset\mathrm{GL}(V)$ can be realized over $\mathrm{Q}(\zeta_{G})$, where $\zeta_{G}$ denotes a primitive $|G|$-th root of unity, in the sense that there is a choice of basis for $V$ with respect to which $G\subset\mathrm{GL}_r(\mathbb{Q}(\zeta_G))$. For $\sigma\in\mathrm{Gal}\bigl(\mathbb{Q}(\zeta_{G})/\mathbb{Q}\bigr)$, the \defn{Galois twist} $V^\sigma$ of $V$ is the representation of $G$ on the same underlying vector space $V$ obtained by applying $\sigma$ to the matrix entries of $g\in\mathrm{GL}_r(\mathbb{Q}(\zeta_{G}))$. Alternatively and equivalently, one can define $V^\sigma$ by applying $\tilde{\sigma}$ to the matrix entries of $g$ in terms of any basis of $V$, for $\tilde{\sigma}$ any extension of $\sigma$ to a field automorphism of $\mathbb{C}$.

In~\cite{orlik1980unitary}, Orlik and Solomon gave the following generalization of \Cref{thm:solomon}.

\begin{theorem}[{\cite[Corollary 3.2]{orlik1980unitary}}]
\[\left(S(V^*)\otimes \bigwedge (V^\sigma)^*\right)^G\simeq S(V^*)^G\otimes\bigwedge \UGsd,\] where the degrees of the homogeneous generators of $\UGsd:=(\CG\otimes \Vsd)^G$ are $e_i^G(\Vs)$, the $\Vs$-exponents of $G$.
\label{thm:os2}
\end{theorem}

\noindent Computing the Poincar\'e series in two ways as in~\Cref{cor:sol1} gives the following formula.

\begin{corollary}[{\cite[Theorem 3.3]{orlik1980unitary}}]
\[\Hilb\left(\left(S(V^*)\otimes \bigwedge (V^\sigma)^* \right)^G;x,u\right) = \frac{1}{|G|} \sum_{g \in G} \frac{\mathrm{det}(1+ug |_{V^\sigma})}{\det(1-x g|_V)} = \prod_{i=1}^r\frac{1+x^{e_i^G(V^\sigma)}u}{1-x_i^{d_i^G}}.\]
\label{cor:sol2}
\end{corollary}

\noindent Specializing~\Cref{cor:sol2} to $u=q(1-x)-1$ and taking the limit as $x \to 1$ gives \Cref{thm:orlik_solomon}.

\subsection{Amenable Representations}
\label{sec:amenable1}
More generally, an $m$-dimensional $G$-module $M$ satisfying $\sum_{i=1}^m e_i^G(M) = e^G_1(\bigwedge^m M)$ is called \defn{amenable}. This amenability condition can be shown to be equivalent to the requirement that $(S(V^*)\otimes\bigwedge M^*)^G$ be a free exterior algebra over $S(V^*)^G$.  Since, in particular, Galois twists $V^\sigma$ of the reflection representation $V$ of $G$ are amenable, the following theorem generalizes~\Cref{thm:os2}.

\begin{theorem}[{\cite[Theorem 3.1]{orlik1980unitary}}]\label{thm:OS_amenable}
Let $M$ be an amenable $G$-module.  Then \[\left(S(V^*)\otimes \bigwedge M^*\right)^G\simeq S(V^*)^G\otimes\bigwedge \UGMd,\] where $\UGMd:=(\CG\otimes M^*)^G$ and the degrees of the homogeneous generators of $\UGMd$ are $e_i^G(M)$, the $M$-exponents of $G$.
\end{theorem}

From this, one can pursue the usual strategy of computing the Poincar\'e series of $(S(V^*)\otimes\bigwedge M^*)^G$ in two different ways to obtain the following.

\begin{corollary} If $M$ is an amenable $G$-module, then
\[\Hilb\left(\left(S(V^*)\otimes \bigwedge M^* \right)^G;x,u\right) = \frac{1}{|G|} \sum_{g \in G} \frac{\mathrm{det}(1+ug |_{M})}{\det(1-x g|_V)} = \frac{\prod_{i=1}^m (1+x^{e_i^G(M)}u)}{\prod_{i=1}^r (1-x_i^{d_i^G})}.\]
\label{cor:amen}
\end{corollary}

\noindent However, it is no longer clear how to specialize~\Cref{cor:amen} in the same way as~\Cref{cor:sol1,cor:sol2} to obtain an analogue of \Cref{eq:shephard_todd} and \Cref{thm:orlik_solomon} in this generality.

\begin{definition}\label{def:OS_space}
Let $M$ be a $G$-module. We define the \defn{Orlik-Solomon space} $\UGM$ to be the dual $G$-module to $\UGMd:=(\CG\otimes M^*)^G$. In the special case where $M=V$, we write $\UG:=U_V^G$.  When $M=V^\sigma$, we write $\UGs:=U_{V^\sigma}^G$.
\end{definition}

\section{Normal Reflection Subgroups}
\label{sec:quotients}

The following theorem is a special case of results in~\cite{bessis2002quotients} (where they consider the more general notion of ``bon sous-groupe distingu\'e'' in lieu of our normal reflection subgroup $N$ of $G$). We emphasize that our proof follows the ideas in~\cite{bessis2002quotients}, specialized to our more restricted setting.

{
\renewcommand{\thetheorem}{\ref{thm:reflection_action}}
\begin{theorem}
Let $G\subset \mathrm{GL}(V)$ be a complex reflection group and let $N\trianglelefteq G$ be a normal reflection subgroup. Then $G/N=H$ acts as a reflection group on the vector space $V/N=\VN$.
\end{theorem}
\addtocounter{theorem}{-1}
}

\begin{proof} We claim that there exist homogeneous generators $\n_1,\dots,\n_r$ of $S(V^*)^N$ such that $\VN^*=\mathrm{span}_\mathbb{C}\{\n_1,\ldots,\n_r\}$ is $H$-stable. By Theorem~\ref{thm:shephard_todd}, the ring of $N$-invariants $S(V^*)^N=\mathbb{C}[\tilde{\n}_1,\dots,\tilde{\n}_r]$ for some homogeneous algebraically independent $\tilde{\n}_i$. Let $I_+\subset S(V^*)^N$ be the ideal generated by homogeneous $N$-invariants of positive degree. Then both $I_+$ and $I_+^2$ are $H$-stable homogeneous ideals, and therefore the algebraic tangent space $I_+/I_+^2$ to $\VN=V/N$ at $0$ inherits a graded action of $H$ that is compatible with the (graded) quotient map $\pi:I_+\twoheadrightarrow I_+/I_+^2$. Hence there exists a graded $H$-equivariant section $\varphi:I_+/I_+^2\rightarrow I_+$. Letting $\n_i=\varphi\circ\pi(\tilde{\n}_i)$ we see that $\n_1,\dots,\n_r$ are still homogeneous algebraically independent generators for $S(V^*)^N$ with $\mathrm{deg}(\n_i)=\mathrm{deg}(\tilde{\n}_i)$ and $\VN^*:=\spn_\CC\{\n_1,\dots,\n_r\}$ is $H$-stable.

Let $\mathbf{x}=\{x_1,\dots,x_r\}$ denote a dual basis for $V$ and $\mathbf{\n}=\{\n_1,\dots,\n_r\}$ denote an $H$-stable basis for $\VNd$ as above. Since the action of $H$ on the polynomial ring $S(\VNd)=S(V^*)^N$ is obtained from the action of $G$ on $S(V^*)$, it preserves $\mathbf{x}$-degrees as well as $\mathbf{N}$-degrees. Therefore we may choose the fundamental $G$-invariants $\g_i(\mathbf{x})\in S(V^*)^G=(S(V^*)^N)^H=S(\VNd)^H$ to be simultaneously $\mathbf{x}$-homogeneous and $\mathbf{N}$-homogeneous, so that $\h_i(\mathbf{N}):=\g_i(\mathbf{x})$ form a set of $\mathbf{N}$-homogeneous generators for the polynomial ring $S(\VNd)^H$. Since any algebraic relation $f(\h_1,\ldots,\h_r)=0$ would result in an algebraic relation $f(\g_1,\ldots,\g_r)=0$, the $\mathbf{N}$-homogeneous $\h_i(\mathbf{N})$ must be algebraically independent. By \Cref{thm:shephard_todd}, $H$ is a complex reflection group.\qedhere
\end{proof}

\begin{remark}
\label{rem:indexing} As pointed out in \cite[Proposition~3.16]{bessis2002quotients} and explained in \cite[Section~8.3]{bonnafe2006twisted}, the action of $H$ on $\VN$ is often not irreducible. Denote by $\DN=\{d_1^N,\dots,d_r^N\}$ the set of degrees $d_i^N:=\mathrm{deg}_\mathbf{x}(\n_i)$. For $d\in\DN$, let us write $\mathbf{N}_d:=\{\n_i\in\mathbf{N} \ | \ \mathrm{deg}_\mathbf{x}(\n_i)=d\}$ and $\VNdd:=\spn_\CC\mathbf{N}_d$, so that $\VNd\simeq\bigoplus_{d\in\DN}\VNdd$ and $S(\VNd)\simeq\bigotimes_{d\in\DN}S(\VNdd)$ as graded $H$-modules. For $d\in\DN$, let $H_{(d)}\subset\mathrm{GL}(\VNdd)$ denote the image of $H$ in $\mathrm{GL}(\VNdd)$, so that $H$ decomposes as a direct product $\bigtimes_{d\in\DN} H_{(d)}$, where each $H_{(d)}$ is a reflection group on the graded dual $\VN_{-d}$ of $\VNdd$. We see that in fact there exist algebraically independent (bi)homogeneous polynomials $\smash{\h_i(\mathbf{N}_{d_i^N})\in S(\VN^*_{d_i^N})}$ such that $S(\VNd)^H=\CC[\h_1,\dots,\h_r]$, so that the fundamental $G$-invariants $\g_1,\dots,\g_r$ generating $S(V^*)^G=(S(V^*)^N)^H=S(E^*)^H$ can be expressed as
\begin{equation}\label{rearrangement-eq} \g_i(\mathbf{x})=\h_i(\mathbf{\n}_{d_i^N}).\end{equation}
Having chosen the fundamental $G$-invariants $\g_i$ to have degrees $d_1^G\leq \dots \leq d_r^G$, we implicitly index the $N$-degrees $d_i^N$ and $H$-degrees $d_i^H$ so that \Cref{rearrangement-eq} is satisfied.
\end{remark}

\begin{remark}
Unlike in the real case~\cite{gal2005normal,bonnafe2010semidirect}, $H$ is not necessarily a reflection subgroup of $G$ or even a subgroup of $G$.  A counterexample is given by $G_8=G \triangleright N=G(4,2,2)$, so that $G/N \simeq \mathfrak{S}_3$---but $\mathfrak{S}_3$ is not a subgroup of $G_8$.
\end{remark}

\subsection{Harmonic polynomials}\label{sec:coinvariants}
The space of $N$-harmonic polynomials $\CN\subset S(V^*)$ is $G$-stable \cite[Proposition~12.2]{lehrer2009unitary} and isomorphic to the regular representation of $N$ \cite[Corollary~9.37]{lehrer2009unitary}. The space of $H$-harmonic polynomials $\CH\subset S(\VNd)$ is bigraded, by $\mathbf{x}$-degree as well as by $\mathbf{N}$-degree, and therefore it admits a $\CC$-basis of $H$-harmonic polynomials that are simultaneously $\mathbf{x}$-homogeneous and $\mathbf{N}$-homogeneous. The following result elaborates on \cite[Corollary~8.4]{bonnafe2006twisted} in our present setting.

\begin{proposition}
\label{prop:coinvariant_decomposition}
There is a graded $G$-module isomorphism $\CG\simeq\CH\otimes\CN$ such that $(\CG)^N\simeq\CH$ as graded $H$-modules.
\end{proposition}

\begin{proof}
Putting together the isomorphisms $S(V^*)\simeq S(\VNd)\otimes \CN$ as graded $N$-modules, $S(\VNd)\simeq S(\VNd)^H\otimes \CH$ as bigraded $H$-modules (equivalently, as bigraded $G$-modules of $N$-invariants), and $S(V^*)^G=S(\VNd)^H$, we obtain the isomorphism $S(V^*)\simeq S(V^*)^G\otimes\CH\otimes\CN$ as graded $G$-modules. Letting $\pi:S(V^*)\rightarrow S(V^*)/\IG$ denote the canonical projection, we see that $\CC\otimes\CH\otimes\CN$ must surject onto the image $S(V^*)/\IG\simeq \CG$, because $S(V^*)^G$ is generated as a $\CC$-algebra by the generators of the ideal $\IG$. But this surjection $\CH\otimes\CN\rightarrow \CG$ of graded $G$-modules must then be an isomorphism, because\[\mathrm{dim}_\CC(\CH\otimes\CN)=\mathrm{dim}_\CC(\CH)\cdot\mathrm{dim}_\CC(\CN)=|H|\cdot|N|=|G|=\mathrm{dim}_\CC(\CG).\]
Since $\CH\subset S(\VNd)$ consists of $N$-invariants, we have $(\CH\otimes\CN)^N=\CH\otimes(\CN)^N=\CH\otimes\CC$, and therefore $(\CG)^N\simeq\CH$ as graded $H$-modules, as claimed.\end{proof}

\begin{remark}
It follows from the isomorphism $S(V^*)\simeq S(V^*)^G\otimes\CG$ as graded $G$-modules that the Poincar\'e series of $\CG$ can be written~\cite[Theorem~B]{chevalley1955invariants}
\[\Hilb(\CG;q) = \frac{\Hilb(S(\Vd);q)}{\Hilb(S(V^*)^G;q)} = \prod_{i=1}^r\frac{1-q^{d_i^G}}{1-q}.\]
Since $|G|=d_1\cdots d_r$, it is natural to ask for a combinatorial interpretation of $\Hilb(\CG;q)$ as a weighted sum over the elements of $G$.  When $G$ is a \emph{real} reflection group, $G$ acts simply transitively on the connected components of its real hyperplane complement. Assigning some base connected component $R_e$ to the identity element $e \in G$ gives a bijection between group elements and regions sending $g \leftrightarrow R_g$, and we can define the statistic $\mathrm{inv}(g)$ to be the number of \defn{inversions} of $g \in G$---that is, the number of hyperplanes separating the connected component $R_g$ from $R_e$. Then the Poincar\'e series of $\CG$ has the well-known interpretation \[\Hilb(\CG;q) = \prod_{i=1}^r\frac{1-q^{d_i^G}}{1-q} = \sum_{g\in G} q^{\mathrm{inv}(g)}.\]  In principle, our \Cref{prop:coinvariant_decomposition} gives a new method for producing such combinatorial interpretations: given interpretations of $\Hilb(\CH;q)$ and $\Hilb(\CN;q)$ and coset representatives $\{h\}$ for $G/N$, we can write \begin{align*}\Hilb(\CG;q) &=\Hilb(\CH \otimes \CN;q) = \Hilb(\CH;q)\cdot\Hilb(\CN;q) \\&= \left(\sum_{hN\in H} q^{\mathrm{stat}_H(hN)}\right)\left(\sum_{n\in N} q^{\mathrm{stat}_N(n)}\right) \\ &= \sum_{hn \in G} q^{\mathrm{stat}_H(hN)+\mathrm{stat}_N(n)}.\end{align*}
\end{remark}

\begin{remark}It follows from \Cref{prop:coinvariant_decomposition} that for any $G$-module $M$ the (dual) Orlik-Solomon space $\UNMd=(\CN\otimes M^*)^N$ as in \Cref{def:OS_space} can be considered as an $H$-module, or (equivalently) as a $G$-module of $N$-invariants.
\end{remark}

The following result is useful in determining the Orlik-Solomon space $\UNM$ in particular examples (cf.~\Cref{sec:examples}) up to graded $G$-module isomorphism.

\begin{lemma}\label{lem:fake_OS_space} Let $\eta_N: S(V^*)\rightarrow\CN$ denote the $G$-equivariant projection onto the space of $N$-harmonic polynomials. Let $M$ be a $G$-module of rank $m$, and suppose that $\tilde{u}_1,\dots,\tilde{u}_m$ form a homogeneous basis for $(S(V^*)\otimes M^*)^N$ as a free $S(V^*)^N$-module such that $(\tilde{U}_M^N)^*:=\spn_\CC\{\tilde{u}_1,\dots,\tilde{u}_m\}$ is $G$-stable. Then the restriction of the projection $\eta_N\otimes 1:(\tilde{U}_M^N)^*\rightarrow \UNMd$ is a graded $G$-module isomorphism.
\end{lemma}

\begin{proof} Let $u_1,\dots,u_m$ be a homogeneous basis for $\UNMd:=(\CN\otimes M^*)^N$. We may assume that $\mathrm{deg}(u_i)=e_i^N(M)=\mathrm{deg}(\tilde{u}_i)$ for each $i=1,\dots,m$. Let $y_1,\dots,y_m$ be a basis for $M^*$, and write $\tilde{u}_i:=\sum_{j=1}^m\tilde{a}_{ij}\otimes y_j$ and $u_i=\sum_{j=1}^ma_{ij}\otimes y_j$, where $\tilde{a}_{ij}\in S(V^*)$ and $a_{ij}\in\CN$, and every non-zero $\tilde{a}_{ij}$ and $a_{ij}$ is homogeneous of degree $e_i^N(M)$. Since the $\tilde{u}_i$ and the $u_i$ form bases for $S(V^*)\otimes M^*)^N$ as a free $S(V^*)^N$-module, there exists a matrix $[p_{ij}]\in\mathrm{GL}_m(S(V^*)^N)$ such that $[\tilde{a}_{ij}]=[p_{ij}]\cdot [a_{ij}]$. Since the kernel of $\eta_N:S(V^*)\rightarrow \CN$ is precisely the ideal generated by homogeneous $N$-invariants of positive degree, it follows that $[\eta_N(\tilde{a}_{ij})]=[p_{ij}(0)]\cdot[a_{ij}]$, where $p_{ij}(0)$ denotes the evaluation at $0\in V$. Since $\mathrm{deg}(\mathrm{det}[\tilde{a}_{ij}])=\sum_{i=1}^m e_i^N(M)=\mathrm{deg}(\mathrm{det}[a_{ij}])$, it follows that $\mathrm{det}[p_{ij}]\in\CC^\times$, and therefore $[p_{ij}(0)]\in\mathrm{GL}_m(\CC)$. Since $\eta_N\otimes 1$ is $G$-equivariant and $(\tilde{U}_M^N)^*$ is $G$-stable, $\eta_N\otimes 1:(\tilde{U}_M^N)^*\rightarrow\UNMd$ is a graded isomorphism of $G$-modules.  \end{proof}

\subsection{Numerology}

The following consequence of \Cref{prop:coinvariant_decomposition} establishes the first equation of \Cref{thm:numbers} in more generality.
\begin{corollary}\label{cor:general_numerology}
Let $M$ be a $G$-module and define $\UNM$ as in \Cref{def:OS_space}. For a suitable choice of indexing we have \[e_i^N(M)+e_i^G(\UNM)=e_i^G(M).\]
\end{corollary}

\begin{proof}Applying \Cref{prop:coinvariant_decomposition}, we see that \[(\CG\otimes M^*)^G\simeq((\CH\otimes\CN\otimes M^*)^N)^H=(\CH\otimes(\CN\otimes M^*)^N)^H=(\CH\otimes \UNMd)^H.\]
Let $m$ be the rank of $M$, and let $y_1,\dots,y_m$ be a basis of $M^*$. Let $a_{ij}^N\in \CN$ be $\mathbf{x}$-homogeneous such that $u_i^N:=\sum_{j=1}^ma_{ij}^N\otimes y_j$ form a basis for $\UNMd$ with $\mathrm{deg}_\mathbf{x}(u_i^N)=e_i^N(M)$. Letting $\ENM:=\{e_1^N(M),\dots,e_m^N(M)\}$ denote the set of $M$-exponents of $N$, we see that \[\UNMd\simeq\bigoplus_{e\in\ENM}\UNMde\] as an $H$-module. Therefore, there exist $\mathbf{x}$-homogeneous $a_{ij}^H\in\CH$ such that $u_i^G:=\sum_{j=1}^ma_{ij}^H\otimes u_j^N$ form a basis for $(\CH\otimes \UNMd)^H$ with $\mathrm{deg}_\mathbf{x}(u_i^G)=e_i^G(U_M^N)$ and such that $a_{ij}^H=0$ whenever $\mathrm{deg}_\mathbf{x}(u_j^N)\neq e_i^N(M)$ (in other words, the matrix $[a_{ij}^H]$ may be chosen to be square-diagonal corresponding to the graded decomposition of $\UNMd$). But then the $u_i^G$ form an $\mathbf{x}$-homogeneous basis for $(\CG\otimes M^*)^G\simeq(\CH\otimes(U_M^N)^*)^H$, and we see that \[e_i^G(M)=\mathrm{deg}_\mathbf{x}(u_i^G)=e_i^G(U_M^N)+e_i^N(M).\qedhere\]
\end{proof}

\subsection{Amenability}\label{sec:amenability}
Recall from \Cref{sec:amenable1} that a $G$-module $M$ of rank $m$ is called \defn{amenable} if $\sum_{i=1}^m e_i^G(M)=e_1^G(\bigwedge^mM)$.

\begin{remark} \label{rem:amenable_equivalence}Regardless of whether a $G$-module $M$ of rank $m$ is amenable, we always have a natural $S(V^*)^G$-linear injective homomorphism \[S(V^*)^G\otimes{\textstyle\bigwedge^m}(\CG\otimes M^*)^G\hookrightarrow S(V^*)^G\otimes(\CG\otimes{\textstyle \bigwedge^m}M^*)^G\] in $(S(V^*)\otimes\bigwedge M^*)^G$, which identifies $\bigwedge^m(\CG\otimes M^*)^G$ with $a\otimes (\CG\otimes\bigwedge^m M^*)^G$ for some $0\neq a\in S(V^*)^G$. The amenability of $M$ as a $G$-module is therefore precisely the requirement that $a\in\mathbb{C}$ be a constant polynomial. For brevity and convenience, we will summarize this equivalent characterization of amenability in the following \Cref{lem:amenable_top_wedge} (cf.~\cite[Theorem~2.10]{bonnafe2006twisted}).
\end{remark}

\begin{lemma}\label{lem:amenable_top_wedge}
A $G$-module $M$ of rank $m$ is amenable if and only if \[{\textstyle\bigwedge^m}((\CG\otimes M^*)^G)\simeq(\CG\otimes{\textstyle\bigwedge^m}M^*)^G.\] 
\end{lemma}

The decomposition $\CG\simeq\CH\otimes\CN$ from \Cref{prop:coinvariant_decomposition} and its \Cref{cor:general_numerology} have many useful consequences. Although the following result can be proved more directly by appealing to \cite[Corollary~8.7]{bonnafe2006twisted}, we provide a full proof.

\begin{lemma} \label{lem:amenability_N_trivial}Let $M$ be an $H$-module. Then:
\begin{enumerate}[(i)]
\item $(\CG\otimes M^*)^G\simeq (\CH\otimes M^*)^H$; and
\item $M$ is amenable as a $G$-module if and only if $M$ is amenable as an $H$-module.
\end{enumerate}
\end{lemma}

\begin{proof} The $M$-exponents $e_i^G(M)$ of $G$ are obtained as the $\mathbf{x}$-degrees of any $\mathbf{x}$-homogeneous basis for $(\CG\otimes M^*)$, and the $M$-exponents $e_i^H(M)$ of $H$ are analogously obtained as the $\mathbf{N}$-degrees of of any (bi)homogeneous basis for $(\CH\otimes M^*)^H$. Letting $m$ denote the rank of $M$, the fake degree $e_1^G(\bigwedge^m M)$ is the $\mathbf{x}$-degree of a basis element for $(\CG\otimes\bigwedge^m M^*)^G$, and the fake degree $e_i^H(\bigwedge^m M)$ is the $\mathbf{N}$-degree of a basis element for $(\CH\otimes \bigwedge^m M^*)^H$. Since both $M$ and $\CH$ are $N$-invariant, and $\CG\simeq\CH\otimes\CN$ as graded $G$-modules such that $(\CG)^N\simeq\CH$ by \Cref{prop:coinvariant_decomposition}, we obtain \[(\CG\otimes M^*)^G=((\CG\otimes M^*)^N)^H= ((\CG)^N\otimes M^*)^H\simeq (\CH\otimes M^*)^H,\] which establishes (i).

Applying (i) to $\bigwedge^mM^*$ we have that $(\CG\otimes{\textstyle\bigwedge^m}M^*)^G\simeq (\CH\otimes{\textstyle\bigwedge^m}M^*)^H$. By \Cref{lem:amenable_top_wedge}, the amenability of $M$ as a $G$-module and the amenability of $M$ as an $H$-module are respectively equivalent to \[{\textstyle\bigwedge^m}((\CG\otimes M^*)^G)\simeq(\CG\otimes{\textstyle\bigwedge^m}M^*)^G\quad\text{and}\quad{\textstyle\bigwedge^m}((\CH\otimes M^*)^H)\simeq(\CH\otimes{\textstyle\bigwedge^m}M^*)^H,\] which proves (ii). \end{proof}

\begin{remark}
Since the Orlik-Solomon space $\UNM$ of \Cref{def:OS_space} is trivial as an $N$-module, it follows from \Cref{lem:amenability_N_trivial} that $\UNM$ is amenable as a $G$-module if and only if it is amenable as an $H$-module. From now on we will just say that $\UNM$ is amenable whenever these equivalent conditions hold.
\end{remark}

\begin{proposition}\label{prop:amenable_equivalence} Let $M$ be a $G$-module and define $\UNM$ as in \Cref{def:OS_space}. Then:

\begin{enumerate}[(i)]
\item If $M$ is amenable as a $G$-module, then $\UNM$ is amenable.
\item If $M$ is amenable as an $N$-module and $\UNM$ is amenable, then $M$ is amenable as a $G$-module.
\end{enumerate}
\end{proposition}

\begin{proof} Let $m$ denote the rank of $M$. Define the amenability defects \begin{gather*}
\gamma:=\sum_{i=1}^m e_i^G(M)-e_1^G({\textstyle\bigwedge^m} M);\qquad
\nu:=\sum_{i=1}^me_i^N(M)-e_1^N({\textstyle\bigwedge^m}M);\qquad \text{and}\\
\eta:=\sum_{i=1}^m e_i^G(U_M^N)-e_1^G({\textstyle\bigwedge^m}U_M^N),
\end{gather*}
so that $M$ is amenable as a $G$-module if and only if $\gamma=0$; $M$ is amenable as an $N$-module if and only if $\nu=0$; and $\UNM$ is amenable as a $G$-module if and only if $\eta=0$. By \Cref{lem:amenability_N_trivial}, $\UNM$ is also amenable as an $H$-module if and only if $\eta=0$. By \cite[Lemma~2.8]{orlik1980unitary} (see also \Cref{rem:amenable_equivalence}), in any case we have that $\gamma,\nu,\eta\geq 0$.

Let $W_M^N$ be the dual of $(W_M^N)^*:=(\CN\otimes{\textstyle\bigwedge^m}M^*)^N$ (instead of the more natural but cumbersome $U_{\bigwedge^m\!\! M}^N$ in lieu of $W_M^N$). From \Cref{cor:general_numerology}, we know that \[e_i^N(M)+e_i^G(U_M^N)=e_i^G(M)\qquad\text{and}\qquad e_1^N({\textstyle\bigwedge^m}M)+e_1^G(W_M^N)=e_1^G({\textstyle\bigwedge^m}M).\] Summing the first equation over $i=1,\dots,m$ and subtracting the second equation, we obtain\begin{equation}\label{eq:amenable_defects}\nu+\eta+e_1^G({\textstyle\bigwedge^m}U_M^N)=\gamma+e_1^G(W_M^N).\end{equation}

Consider now the natural inclusion of graded $G$-modules \[S(\VNd)\otimes{\textstyle\bigwedge^m}(\CN\otimes M^*)^N\hookrightarrow S(\VNd)\otimes{\textstyle}(\CN\otimes{\textstyle\bigwedge^m}M^*)^N\] in $(S(V^*)\otimes\bigwedge M^*)^N$, which identifies $\bigwedge^m(U_M^N)^*$ with $a\otimes (W_M^N)^*$ as graded $G$-modules for some $\mathbf{x}$-homogeneous $a\in S(\VNd)$ with $\mathrm{deg}_\mathbf{x}(a)=\nu$ (cf.~\Cref{rem:amenable_equivalence}). From this it follows that $e_1^G(\bigwedge^m U_M^N)+\nu\geq e_1^G(W_M^N)$, which together with \Cref{eq:amenable_defects} implies that $\eta\leq\gamma$. Therefore if $M$ is amenable as a $G$-module then $U_M^N$ is amenable. If $M$ is amenable as an $N$-module, so that $\nu=0$, then we see that $\bigwedge^m\UNMd\simeq (W_M^N)^*$ as $G$-modules. Therefore $e_1^G(\bigwedge^m\UNM)= e_1^G(W_M^N)$, and we obtain from \Cref{eq:amenable_defects} that $\eta=\gamma$ in this case. Hence if $M$ is amenable as an $N$-module and $\UNM$ is amenable, then $M$ is amenable as a $G$-module.\end{proof}

\begin{remark} It is not true in general that $M$ being amenable as a $G$-module implies that $M$ is amenable as an $N$-module. For a counterexample, let $G=C_a=\langle c\rangle$ and $N=C_d=\langle c^e\rangle$ with $a=de$, acting on $V=\mathbb{C}$ in the standard reflection representation by $c\mapsto \zeta_a$, a primitive $a$-th root of unity. Consider the $G$-module $M:=V^*\oplus (V^*)^{\otimes(d-1)}$, so that $\bigwedge^2M\simeq (V^*)^{\otimes d}$. Then $e_1^G(M)=1=e_1^N(M)$ and $e_2^G(M)=d-1=e_2^N(M)$. Since $e_1^G(\bigwedge^2 M)=d=e_1^G(M)+e_2^G(M)$, $M$ is amenable as a $G$-module. However, $\bigwedge^2M$ is trivial as an $N$-module, and therefore $e_1^N(\bigwedge^2 M)=0\neq d=e_1^N(M)+e_2^N(M)$, so $M$ is not amenable as an $N$-module.\end{remark}

\subsection{Poincar\'e Series} Suppose $M$ is an amenable $G$-module of rank $m$. Then by \Cref{prop:amenable_equivalence} and \Cref{lem:amenability_N_trivial}, the Orlik-Solomon space $\UNM$ of \Cref{def:OS_space} is amenable (considered either as a $G$-module or as an $H$-module). Let us again write $\ENM:=\{e_1^N(M),\dots,e_m(M)\}$ for the set of $M$-exponents of $N$. We have a graded $G$-module decomposition \[\UNMd=\bigoplus_{e\in\ENM}\UNMde,\] from which we obtain more generally a graded $G$-module decomposition of $\bigwedge^p\UNMd$: \[\bigl({\textstyle\bigwedge^p}\UNMd\bigr)_e:=\spn_\CC\left\{u_{i_1}\wedge\dots\wedge u_{i_p} \ \middle| \ u_{i_j}\in \bigl(U_N^M\bigr)^*_{e_j} \ \text{with} \ {\textstyle\sum_{j=1}^p} e_j=e\right\}.\] This results in an obvious bigrading of $\bigwedge\UNMd$ and a corresponding tri-grading of the associative algebra $S(\Vd)\otimes \bigwedge\UNMd$.

\begin{definition}\label{def:trigraded_poincare}The tri-graded Poincar\'e series for $(S(V^*)\otimes\bigwedge \UNMd)^G$ is \[\P^G_M(x,y,u):=\sum_{\ell,e,p\geq 0}\mathrm{dim}_\CC\left(S(V^*)_\ell\otimes\left({\textstyle\bigwedge^p}\UNMd\right)_e\right)^G  x^\ell y^e u^p.\]\end{definition}

We will follow the usual strategy of computing this Poincar\'e series in two different ways to deduce combinatorial formulas. Let us denote as before $d_1^G,\dots,d_r^G$ the degrees of the fundamental $G$-invariants generating $S(V^*)^G$ as a polynomial algebra. Let us index the $M$-exponents of $N$, $e_1^N(M),\dots,e_m^N(M)$, and the $\UNM$-exponents of $G$, $e_1^G(\UNM),\dots,e_m^G(\UNM)$, as in \Cref{cor:general_numerology}, so that $e_i^N(M)+e_i^G(\UNM)=e_i^G(M)$.

\begin{proposition}\label{prop:general_product_side}
$\P^G_M(x,y,u)=\displaystyle\frac{\prod\limits_{i=1}^m\left(1+x^{e_i^G(\UNM)}y^{e_i^N(M)}u\right)}{\prod\limits_{j=1}^r\left(1-x^{d_j^G}\right)}.$
\end{proposition}

\begin{proof}We proceed as in the proof of \Cref{cor:general_numerology}: let $y_1,\dots,y_m$ be a basis of $M^*$. Let $a_{ij}^N\in \CN$ be $\mathbf{x}$-homogeneous such that $u_i^N:=\sum_{j=1}^ma_{ij}^N\otimes y_j$ form a basis for the $e_i^N(M)$-homogeneous component of $\UNMd$, and choose $\mathbf{x}$-homogeneous $a_{ij}^H\in\CH$ such that $u_i^G:=\sum_{j=1}^ma_{ij}^H\otimes u_j^N$ form a basis for $(\CH\otimes (U_M^N)^*_{e_i^N(M)})^H$ with $\mathrm{deg}_\mathbf{x}(u_i^G)=e_i^G(U_M^N)$ (where again $a_{ij}^H=0$ whenever $\mathrm{deg}_\mathbf{x}(u_j^N)\neq e_i^N(M)$). But then the $u_i^G$ form an $\mathbf{x}$-homogeneous basis for $(\CH\otimes(U_M^N)^*)^H$. Since $\UNM$ consists of $N$-invariants, by \Cref{lem:amenability_N_trivial} we have that $(\CG\otimes\UNMd)^G\simeq (\CH\otimes \UNMd)^H$.
By \Cref{prop:amenable_equivalence}, since $M$ is amenable as a $G$-module, $\UNM$ is amenable. Hence, by \Cref{thm:OS_amenable} we have that \[(S(V^*)\otimes{\textstyle\bigwedge}\UNMd)^G\simeq S(V^*)^G\otimes{\textstyle \bigwedge}(\CG\otimes\UNMd)^G.\] Since $(\CG\otimes\UNMd)^G\simeq\spn_\CC\{u_1^G,\dots,u_m^G\}$ and $u_i^G\in S(V^*)_{e_i^G(\UNM)}\otimes (U_M^N)^*_{e_i^N(M)}$, our result follows.\end{proof}

To simplify notation, for $e\in\ENM$ we denote by $\UNMe$ the homogeneous component of $\UNM$ corresponding to the dual of $\UNMde$, rather than the more natural but cumbersome graded dual $(\UNM)_{-e}$ instead of our $\UNMe$.

\begin{proposition}\label{prop:general_sum_side}
$\P^G_M(x,y,u)=\displaystyle\frac{1}{|G|}\sum\limits_{g\in G}\frac{{\textstyle\prod_{e\in\ENM}}\mathrm{det}\left(1+y^eug|(\UNM)_e\right)}{\mathrm{det}\left(1-xg|V\right)}.$
\end{proposition}

\begin{proof} For $g\in G$, let us write \[\P^g_M(x,y,u)=\sum_{\ell,e,p\geq 0}\mathrm{tr}\bigl(g|S(V^*)_\ell\otimes({\textstyle\bigwedge^p}\UNMd)_e\bigr)x^\ell y^e u^p,\] so that $\P^G_M(x,y,u)=\frac{1}{|G|}\sum_{g\in G}\P^g_M(x,y,u)$ and \[\P^g_M(x,y,u)=\left(\sum_{\ell\geq 0}\mathrm{tr}\bigl(g|S(V^*)_\ell\bigr)x^\ell\right)\left(\sum_{e,p\geq 0}\mathrm{tr}\bigl(g|({\textstyle\bigwedge^p}\UNMd)_e\bigr)y^eu^p\right).\]We know that $ \sum_{\ell\geq 0}\mathrm{tr}\bigl(g|S(V^*)_\ell\bigr)x^\ell=\mathrm{det}(1-xg^{-1}|V)^{-1}$. On the other hand, since $\UNMd\simeq\bigoplus_{e\in\ENM}\UNMde$, we have that $\bigwedge \UNMd\simeq\bigotimes_{e\in\ENM}\bigwedge\UNMde$ as bigraded $G$-modules. Hence, for each $g\in G$, \[\sum_{e,p\geq 0}\mathrm{tr}\bigl(g|({\textstyle\bigwedge^p}\UNMd)_e\bigr)y^eu^p=\prod_{e\in\ENM}\left(\sum_{p\geq 0}\mathrm{tr}\bigl(g|{\textstyle\bigwedge^p}\UNMde\bigr)y^{ep}u^p\right).\] For each $e\in \ENM$ we have that $\sum_{p\geq 0}\mathrm{tr}\bigl(g|{\textstyle\bigwedge^p}\UNMde\bigr)y^{ep}u^p=\mathrm{det}(1+y^eug^{-1}|\UNMe)$. Hence, for each $g\in G$, \[ \P^g_M(x,y,u)=\frac{\prod_{e\in\ENM}\mathrm{det}(1+y^eug^{-1}|\UNMe)}{\mathrm{det}(1-xg^{-1}|V)}.\] Our result follows after taking the average over $g\in G$ on each side.\end{proof}


\section{Proofs of the Main Theorems}
\label{sec:proofs}
We are now in a position to apply the results of \Cref{sec:quotients} to prove the main results announced in the introduction. Fix $G\subset\mathrm{GL}(V)$ a complex reflection group acting by reflections on the vector space $V$ of dimension $r$. Let $N\trianglelefteq G$ be a normal reflection subgroup with quotient $H=G/N$, which acts by reflections on $\VN=V/N$. For $\sigma\in\mathrm{Gal}(\mathbb{Q}(\zeta_G)/\mathbb{Q})$, where $\zeta_G$ denotes a primitive $|G|$-th root of unity, write $\Vs$ for the \defn{Galois twist} of $V$ (as defined in \Cref{sec:o-s_thm}). As in \Cref{def:OS_space}, we write $\UN$ for the dual of $(\CN\otimes V^*)^N$, and more generally $\UNs$ for the dual of $(\CN\otimes\Vsd)^N$.

\subsection{Proof of~\texorpdfstring{\Cref{thm:numbers}}{Theorem \ref{thm:numbers}}}

{\renewcommand{\thetheorem}{\ref{thm:numbers}}
\begin{theorem}
Let $G\subset \mathrm{GL}(V)$ be a complex reflection group and let $N\trianglelefteq G$ be a normal reflection subgroup. Let $H=G/N$ and $\VN=V/N$. Then for a suitable choice of indexing we have
\begin{align*}
e_i^N(\Vs) {+} e_i^G(\UNs) & = e_i^G(\Vs)\\
d_i^N \cdot e_i^H(\VNs) & = e_i^G(\VNs)\\
d_i^N\cdot d_i^H & =d_i^G.
\end{align*}
\end{theorem}
\addtocounter{theorem}{-1}}

\begin{proof} The first equality is \Cref{cor:general_numerology} applied to the $G$-module $M=\Vs$, and the last equality follows from the observations in~\Cref{rem:indexing}.

Let us establish the second equality. Let $\n^\sigma_1,\dots,\n^\sigma_r$ denote a basis for $\VNsd$ as a $G$-module. We will show that there exist $a_{ij}\in S(\VNd)$ such that $u_i:=\sum_{j=1}^r a_{ij}\otimes \n_j^\sigma$ form an $\mathbf{N}$-homogeneous basis for $(\CH\otimes \VNsd)^H$ (so that each non-zero $a_{ij}$ is $\mathbf{N}$-homogeneous of $\mathbf{N}$-degree $e_i^H(\VNs)$), and moreover $a_{ij}=0$ whenever $d_j^N\neq d_i^N$ and each $a_{ij}\in S(\mathbf{N}_{d_i^N})$, where as before $\mathbf{N}_d$ denotes the set of fundamental $N$-invariants of degree $d$. Since $(\CG\otimes\VNsd)^G\simeq(\CH\otimes\VNsd)^H$ by \Cref{lem:amenability_N_trivial}, the existence of such $a_{ij}$ will establish our claim, since each non-zero $a_{ij}$ as above will then be $\mathbf{x}$-homogeneous of $\mathbf{x}$-degree $e_i^G(\VNs)= d_i^N\cdot e_i^H(\VNs)$.

As in \Cref{rem:indexing}, let us write $\DN=\{d_1^N,\dots,d_r^N\}$ for the set of degrees of $N$ and $\VNdd$ for the graded component of $\VNd$ spanned by the fundamental $N$-invariants of degree $d\in\DN$. To simplify notation, let us write $\VN_d$ for the graded dual of $\VNdd$, instead of $\VN_{-d}$, so that $\VN\simeq\bigoplus_{d\in\DN}\VN_d$, and similarly the Galois twist $\VNs\simeq\bigoplus_{d\in\DN}\VN^\sigma_d$. We saw in \Cref{rem:indexing} that $H$ decomposes as a direct product $\bigtimes_{d\in\DN}H_{(d)}$, where each $H_{(d)}$ is a reflection group acting on $\VN_d$ and $H_{(d)}$ acts trivially on $\VN^\sigma_{d'}$ whenever $d\neq d'$. We then see that $S(\VNd)\simeq\bigotimes_{d\in\DN} S(\VNdd)$ and each $S(\VNdd)\simeq S(\VNdd)^{H_{(d)}}\otimes \mathcal{C}_{H_{(d)}}$ as $H$-modules, so that in particular $\CH\simeq\bigotimes_{d\in\DN}\mathcal{C}_{H_{(d)}}$ as $H$-modules, where again $H_{(d)}$ acts trivially on $\mathcal{C}_{H_{(d')}}$ for $d\neq d'$. It follows from the above observations that\[(\CH\otimes\VNsd)^H\simeq\left(\bigotimes_{d\in\DN}\mathcal{C}_{H_{(d)}}\otimes\bigoplus_{d\in\DN}(\VN_d^\sigma)^*\right)^{\bigtimes_{d\in\DN}H_{(d)}}\simeq\bigoplus_{d\in\DN}(\mathcal{C}_{H_{(d)}}\otimes(\VN^\sigma_d)^*)^{H_{(d)}},\] so that we may indeed choose $a_{ij}\in\mathcal{C}_{H_{(d_i^N)}}\subset S(\VN_{d_i^N}^*)$ such that the $\mathbf{N}$-homogeneous $u_i=\sum_{j=1}^ra_{ij}\otimes\n_j^\sigma\in(\mathcal{C}_{H_{(d_i^N)}}\otimes (\VN^\sigma_{d_i^N})^*)^{H_{(d_i^N)}}$ (i.e., with $a_{ij}=0$ whenever $d_i^N\neq d_j^N$) form a basis for $(\CH\otimes\VNsd)^H$ that is simultaneously $\mathbf{N}$-homogeneous of $\mathbf{N}$-degree $e_i^H(\VNs)$ and $\mathbf{x}$-homogeneous of $\mathbf{x}$-degree $e_i^G(\VNs)=d_i^N\cdot e_i^H(\VNs)$.
\end{proof}

When $\sigma=1$, the $G$-module $\UNs$ in \Cref{def:OS_space} admits a more concrete description.

\begin{lemma}[{\cite[Example 2.4]{bonnafe2006twisted}}] \label{EU} Let $\eta:S(V^*)\rightarrow\CN$ denote the projection onto the space of $N$-harmonic polynomials, and let \begin{align*} d:\VN^*=\spn_\CC\{\n_1,\dots,\n_r\}&\to \spn_\CC\{d\n_1,\dots,d\n_r\}\\ \n_i&\mapsto d\n_i=\sum_{j=1}^r\frac{\partial \n_i}{\partial x_j}\otimes x_j.\end{align*} Then $(\eta\otimes 1)\circ d:\VNd\rightarrow \UNd$ is a graded (degree $-1$) isomorphism of $G$-modules.
\label{lem:isomorphism}
\end{lemma}

\begin{remark}\label{rem:untwisted_numbers}
As mentioned in the introduction, in the case where $\sigma=1$, once we replace $e_i^G(\UN)=e_i^G(\VN)$ by \Cref{EU}, the equalities in \Cref{thm:numbers} are compatible with the classical relations $d_i^G=e_i^G(V)+1$; $d_i^N=e_i^N(V)+1$; and $d_i^H=e_i^H(\VN)+1$. To see this, we proceed as in \cite[Theorem~1.3]{arreche2020normal}. We found in \Cref{rem:indexing} a choice of $\mathbf{N}$-homogeneous $H$-invariants $\h_i(\mathbf{N}_{d_i^N})=\g_i(\mathbf{x})$, a set of fundamental $G$-invariants as in \Cref{rearrangement-eq}, immediately resulting in the equality $d_i^N\cdot d_i^H=d_i^G$ of \Cref{thm:numbers}. Let us show that this same choice of indexing results in the other two equalities of \Cref{thm:numbers}. We begin by comparing $\mathbf{x}$-degrees in \[d\g_i=\sum_{j=1}^r\frac{\partial\g_i}{\partial x_j}\otimes x_j=\sum_{k=1}^r\frac{\partial \h_i}{\partial \n_k}\cdot d\n_k=\sum_{k=1}^r\sum_{j=1}^r\frac{\partial \h_i}{\partial \n_k}\cdot\frac{\partial\n_k}{\partial x_j}\otimes x_j.\] Recall that $e_i^G(V)=d_i^G-1=\mathrm{deg}_\mathbf{x}(d\g_i)$ and $e_i^N(V)=d_i^N-1=\mathrm{deg}_\mathbf{x}(d\n_i)$. Similarly, $e_i^H(E)=d_i^H-1=\mathrm{deg}_\mathbf{N}(d\h_i)$, where now $d\h_i=\sum_{k=1}^r\frac{\partial \h_i}{\partial \n_k}\otimes\n_k\in  (S(\VNd)\otimes \VNd)^H$. Since $\frac{\partial \h_i}{\partial \n_k}=0$ whenever $\mathrm{deg}_\mathbf{x}(\n_k)\neq d_i^N$, it follows that \[e_i^G(V)=e_i^N(V)+d_i^N\cdot (d_i^H-1)=e_i^N(V)+d_i^N\cdot e_i^H(E).\] It remains to show that $d_i^N\cdot e_i^H(E)=e_i^G(\VN)$ under this same choice of indexing.

The $e_i^G(\VN)$ are the $\mathbf{x}$-degrees of a homogeneous basis for $(\CG\otimes\VNd)^G$. By \Cref{lem:amenability_N_trivial}, $(\CG\otimes\VNd)^G\simeq(\CH\otimes\VNd)^H$, and therefore the $(\eta_H\otimes 1)(d\h_i)$ serve as a homogeneous basis for $(\CG\otimes\VNd)^G$, where $\eta_H:S(\VNd)\rightarrow\CH$ denotes the projection onto the space of $H$-harmonic polynomials (cf.~\Cref{EU}). Hence for any $k$ such that $\frac{\partial \h_i}{\partial \n_k}\neq 0$, the $e_i^G(\VN)$ are given by the $\mathrm{deg}_\mathbf{x}(\frac{\partial \h_i}{\partial \n_k})=(d_i^H-1)d_i^N$.

The proof of the second equality of \Cref{thm:numbers} generalizes what is essential in the case $\sigma=1$---where we have explicit bases for the relevant Orlik-Solomon spaces of \Cref{def:OS_space} in terms of fundamental invariants (up to a harmless isomorphism as in \Cref{EU}): $\{d\g_i\}$ for $\UGd$; $\{d\n_i\}$ for $\UNd$; and $\{d\h_i\}$ for $(U^H)^*$---to the more general situation where $\sigma\in\mathrm{Gal}(\mathbb{Q}(\zeta_G)/\mathbb{Q})$ is arbitrary.
\end{remark}

\begin{example}
Take $G=W(F_4)=G_{28}$, $N$ to be the normal subgroup generated by the reflections corresponding to short roots (see the proof of \Cref{thm:repeated_normal} for more details), and $\sigma=1$. Then $N\simeq W(D_4)$ and $G/N\simeq W(A_2)=\mathfrak{S}_3$ acts by reflections on $\CC \oplus \CC \oplus \CC^2$ (trivially on $\CC\oplus\CC$).  \Cref{thm:numbers} corresponds to the identities
\begin{align*}
(1,5,3,3) {+} (0,0,4,8) &= (1,5,7,11) \\
(2,6,4,4) \cdot (0,0,1,2) & = (0,0,4,8)\\
(2,6,4,4) \cdot (1,1,2,3) & = (2,6,8,12).
\end{align*}
Note that the exponents and degrees of $N \simeq W(D_4)$ must be reordered for the identities to hold (see~\Cref{rem:indexing}).
\end{example}

\begin{remark}\label{rem:E_aint_UNs}

When $\sigma=1$, $\UNs \simeq \VNs$ as $G$-representations by~\Cref{lem:isomorphism}---but it is not always the case that $\UNs \simeq \VNs$ as $G$-representations for more general $\sigma$. For example, take the cyclic groups $G=C_a \triangleright C_d=N$ for $d|a$, with $\sigma$ being complex conjugation. Then $\UNsd = \spn_\mathbb{C}\{ x \otimes x^\sigma\}$, on which $G$ acts trivially. We discuss this in more detail in \Cref{sec:cyclic_examples}. See~\Cref{sec:examples} for more examples of explicit identifications of the spaces $\UNs$.
\end{remark}

\subsection{Proof of~\texorpdfstring{\Cref{thm:main_theorem}}{Theorem \ref{thm:main_theorem}}}

{\renewcommand{\thetheorem}{\ref{thm:main_theorem}}
\begin{theorem}
Let $G\subset \mathrm{GL}(V)$ be a complex reflection group of rank $r$ and let $N\trianglelefteq G$ be a normal reflection subgroup. Let $\VN=V/N$ and $\sigma\in\mathrm{Gal}(\mathbb{Q}(\zeta_G)/\mathbb{Q})$. Then for a suitable choice of indexing we have
\[\sum_{g \in G}\left( \prod_{\lambda_i(g) \neq 1} \frac{1-\lambda_i(g)^\sigma}{1-\lambda_i(g)}\right) q^{\fix_V (g)} t^{\fix_{E} (g)} = \prod_{i=1}^r \left(qt+e_i^N(V^\sigma) t + e_i^G(\UNs)\right),\] %
where the $\lambda_i(g)$ are the eigenvalues of $g\in G$ acting on $V$.
\end{theorem}
\addtocounter{theorem}{-1}
}

We refer to the left-hand side of~\Cref{thm:main_theorem} as the \defn{sum side}, and to the right-hand side as the \defn{product side}. We will prove \Cref{thm:main_theorem} by computing the limit as $x \to 1$ of the specialization $y\mapsto x^t$ and $u\mapsto qt(1-x)-1$ of the tri-graded Poincar\'e series $\P^G_\sigma(x,y,u):=\P^G_{\Vs}(x,y,u)$ from \Cref{def:trigraded_poincare} in two different ways to obtain the sum side and the product side separately.

\begin{proof}[Proof of~\Cref{thm:main_theorem}]
By \Cref{cor:prod_side_specialization,cor:sum_side_specialization} below, both sides are equal to $\lim_{x\to 1} \ |G|\cdot\P^G_{\sigma}\Bigl(x,x^t,qt(1-x)-1\Bigr)$.
\end{proof}

Since $M=\Vs$ is amenable as a $G$-module (and as an $N$-module) by \cite[Thm.~2.13]{orlik1980unitary}, we can apply both \Cref{prop:general_product_side,prop:general_sum_side} in this case to obtain \[\displaystyle\frac{1}{|G|}\sum\limits_{g\in G}\frac{{\textstyle\prod_{e\in\ENs}}\mathrm{det}\left(1+y^eug|(\UNs)_e\right)}{\mathrm{det}\left(1-xg|V\right)}=\P^G_\sigma(x,y,u)=\prod_{i=1}^r\displaystyle\frac{\left(1+x^{e_i^G(\UNs)}y^{e_i^N(\Vs)}u\right)}{\left(1-x^{d_i^G}\right)},\] where $\ENs:=\mathcal{E}_{\Vs}=\{e_1^N(\Vs),\dots,e_r^N(\Vs)\}$ denotes the set of $\Vs$-exponents of $N$ as before. The product side of \Cref{thm:main_theorem} follows immediately.

\begin{corollary}[Product side specialization]
\[\lim_{x\to 1} \ |G|\cdot\P^G_\sigma\Bigl(x,x^t,qt(1-x)-1\Bigr)=\prod_{i=1}^r \left(qt+e_i^N(\Vs) t + e_i^G(\UNs)\right).\]
\label{cor:prod_side_specialization}
\end{corollary}
\begin{proof}
We compute:
\begin{align*}
&|G|\prod_{i=1}^r\frac{1+x^{e_i^G(\UNs) }y^{e_i^N(\Vs)}u}{1-x^{d_i^G}}\Bigg|_{\substack{y=x^t \\ u=qt(1-x)-1\\ x\to 1}}  \\&=
|G|\lim_{x \to 1} \prod_{i=1}^r \left(\frac{x^{e_i^G(\UNs)+t e_i^N(\Vs) }qt(1-x)}{1-x^{d_i^G}}+\frac{1-x^{e_i^G(\UNs)+t e_i^N(\Vs)}}{1-x^{d_i^G}}\right)\\
&=\prod_{i=1}^r \left(qt+e_i^N(\Vs) t + e_i^G(\UNs)\right).\qedhere
\end{align*}
\end{proof}

Our argument for the sum side of \Cref{thm:main_theorem} is more delicate. The reason for this is that for $g\in G$ the fixed space of $g$ acting on $\UNs$ often has larger dimension than the fixed space of $g$ acting on $\Vs$, which causes many terms in the term-by-term limit to be zero.

It turns out, as we will now show, that the contributions \emph{are} correct when taken coset-by-coset. For this, let us define for each coset $Ng\in H=G/N$ the \defn{twisted Poincar\'e series}\[\P^{Ng}_\sigma(x,y,u):=\frac{1}{|N|}\sum_{\substack{n\in N \\ \ell,e,p\geq 0}}\mathrm{tr}(ng|(S(V^*)_\ell\otimes({\textstyle\bigwedge^p}\UNsd)_e))x^\ell y^eu^p,\] so that $\P^G_\sigma(x,y,u)=\frac{1}{|H|}\sum_{Ng\in H}\P^{Ng}_\sigma(x,y,u)$. The following result is proved along the same lines as \Cref{prop:general_sum_side} and serves as an equivalent definition of $\P^{Ng}_\sigma(x,y,u)$.

\begin{lemma}\label{lem:twisted_poincare_sum_1}
$\displaystyle\P^{Ng}_\sigma(x,y,u)=\frac{1}{|N|}\sum_{n \in N} \frac{\prod_{e\in\ENs} \det\left(1+u y^{e} (ng)|_{\UNsde}\right) }{\det(1-x(ng)|_{\Vd})}.$
\end{lemma}

\begin{proof} For each $n\in N$, \begin{multline*}\sum_{\ell,e,p\geq 0}\mathrm{tr}(ng|(S(V^*)_\ell\otimes({\textstyle\bigwedge^p}\UNsd)_e))x^\ell y^eu^p\\ =\left(\sum_{\ell\geq 0}\mathrm{tr}(ng|S(V^*)_\ell)x^\ell\right)\cdot\prod_{e\in\ENs}\left(\sum_{p\geq 0}\mathrm{tr}(ng|{\textstyle\bigwedge^p}\UNsde)y^{ep}u^p\right),\end{multline*} and our result follows from taking the average over $n\in N$.\end{proof}

\begin{definition}\label{def:multiset}We will adopt the following notation for the rest of this section. Let $g\in G$. We will denote by $\bar{\lambda}_1(g),\dots,\bar{\lambda}_r(g)$ the set of eigenvalues of $g$ on $\Vd$. We choose once and for all: a $g$-eigenbasis of fundamental $N$-invariants $\n_i\in \VNd$ such that $\mathrm{deg}_\mathbf{x}(\n_i)=d_i^N$ and $g\n_i=\epsilon_i^g(\VN)\n_i$; and a $g$-eigenbasis for the Orlik-Solomon space (see \Cref{def:OS_space}) $u_i^N\in\UNsd$ such that $\mathrm{deg}_\mathbf{x}(u_i^N)=e_i^N(\Vs)$ and $gu_i^N=\epsilon_i^g(\UNs)u_i^N$. We observe as in \cite{bonnafe2006twisted} that the multisets of pairs\[\left\{\Bigl(\epsilon_i^g(E),d_i^N\Bigr) \ \middle| \ i=1,\dots,r\right\}\qquad\text{and} \qquad\left\{\Bigl(\epsilon_i^g(\UNs),e_i^N(\Vs)\Bigr) \ \middle| \ i=1,\dots, r\right\}\] depend only on $\sigma$ and the coset $Ng\in H$, and not on the choice of coset representative $g\in Ng$.
\end{definition}

\begin{proposition}
\label{prop:twisted_graded-sigma}
$\P^{Ng}_\sigma(x,y,u) = \displaystyle\prod_{i=1}^r \frac{1+ \epsilon_i^g(\UNs) u y^{e_i^N(\Vs)}}{1-\epsilon_i^g(\VN) x^{d_i^N}}.$
\end{proposition}

\begin{proof}
We write $\P^{Ng}_\sigma(x,y,u)$ as in \Cref{lem:twisted_poincare_sum_1}. First observe that, since $\UNs$ is $N$-invariant, for any $n\in N$ we have that \[\prod_{e\in\ENs}\mathrm{det}(1+uy^{e}(ng|_{\UNsde})=\prod_{i=1}^r(1+\epsilon_i^g(\UNs)uy^{e_i^N(\Vs)}),\]
independently of $n\in N$.

Let $\mathcal{D}_N=\{d_1^N,\dots,d_r^N\}$ denote the set of degrees for $N$, and let $\VN^*_d=\spn_\CC\mathbf{N}_d$ for $d\in\mathcal{D}_N$, where as before $\mathbf{N}_d$ denotes the set of fundamental $N$-invariants having degree $d$. Since $S(\VN^*)\simeq\bigotimes_{d\in\mathcal{D}_N}S(\VN^*_d),$ we have that
\[\sum_{\ell\geq 0}\mathrm{tr}\bigl(g|S(\VN^*)_\ell\bigr)x^\ell = \prod_{d\in\mathcal{D}_N}\Bigl(\sum_{\ell\geq 0}\bigl(\mathrm{tr}(g|\mathrm{Sym}^\ell(\VN^*_d))(x^{d})^\ell\bigr), \] where $S(\VN^*)_\ell:=S(\VN^*)\cap \mathrm{Sym}^\ell(V^*)=\mathrm{Sym}^\ell(V^*)^N$, and $\mathrm{Sym}^\ell(V^*)$ denotes the $\ell$-th symmetric power of $V^*$. On the other hand,
\[\prod_{d\in\mathcal{D}_N}\Bigl(\sum_{\ell\geq 0}\mathrm{tr}\bigl(g|\mathrm{Sym}^\ell(\VN^*_{d})\bigr)(x^{d})^\ell\Bigr)=\prod_{d\in\mathcal{D}_N}\frac{1}{\mathrm{det}(1-x^{d}(g|_{\VN^*_{d}}))}=\prod_{i=1}^r\frac{1}{1-\epsilon_i^g(\VN)x_i^{d_i^N}}.\]
Therefore, \[\sum_{\ell\geq 0}\mathrm{tr}\bigl(g|S(\VN^*)_\ell\bigr)x^\ell=\prod_{i=1}^r\frac{1}{1-\epsilon_i^g(\VN)x_i^{d_i^N}}.\]

Since, for $n\in N$, \[\sum_{\ell\geq 0}\mathrm{tr}\bigl(ng|\mathrm{Sym}^\ell(V^*)\bigr)x^\ell = \frac{1}{\mathrm{det}(1-x(ng|_{V^*}))},\] it remains to show that \[\frac{1}{|N|}\sum_{n\in N}\Bigl(\sum_{\ell\geq 0}\mathrm{tr}\bigl(ng|\mathrm{Sym}^\ell(V^*)\bigr)x^\ell\Bigr)=\sum_{\ell\geq 0}\mathrm{tr}\bigl(g|S(\VN^*)_\ell\bigr)x^\ell,\] or equivalently that for each $\ell\geq 0$ we have that \[\frac{1}{|N|}\sum_{n\in N}\Bigl(\mathrm{tr}\bigl(ng|\mathrm{Sym}^\ell(V^*)\bigr)\Bigr)=\mathrm{tr}\bigl(g|S(\VN^*)_\ell\bigr).\]
To see this, note that the operator on $\mathrm{Sym}^\ell(V^*)$ given by \[\frac{1}{|N|}\sum_{n\in N}ng=g\cdot\left(\frac{1}{|N|}\sum_{n\in N}n\right)=g\circ\mathrm{pr}^N_\ell,\] where $\mathrm{pr}^N_\ell=\frac{1}{|N|}\sum_{n\in N} n$ is the projection from $\mathrm{Sym}^\ell(V^*)$ onto its $g$-stable subspace $\mathrm{Sym}^\ell(V^*)^N=S(E^*)_\ell$, whence $\mathrm{tr}\bigl((g\circ\mathrm{pr}^N_\ell)|\mathrm{Sym}^\ell(V^*)\bigr)=\mathrm{tr}\bigl(g|S(E^*)_\ell\bigr)$. \end{proof}

\begin{proposition}[{\cite[Theorem 3.1]{bonnafe2006twisted}}]
\label{thm:twisted_os} 
\[\frac{1}{|N|}\sum_{n \in N} \frac{\det(1+u (ng)|_{\Vsd})}{\det(1-x (ng)|_{\Vd})} = \prod_{i=1}^r \frac{1+ \epsilon_{i}^{g}(\UNs) u x^{e_i^N(\Vs)}}{1-\epsilon_{i}^{g}(\VN) x^{d_i^N}}=\P^{Ng}_\sigma(x,x,u),\]
\end{proposition}

\begin{proof}
The first equality is a special case of \cite[Theorem 3.1]{bonnafe2006twisted} (where we note that the change in sign from the $-u$ in their notation to $+u$ in our notation is harmless), and the second equality follows directly from \Cref{prop:twisted_graded-sigma}.
\end{proof}

We obtain the following crucial specialization of \Cref{prop:twisted_graded-sigma}, which exploits the similarity between~\Cref{thm:twisted_os} and~\Cref{prop:twisted_graded-sigma}, and is inspired by \cite[Theorem~3.3]{bonnafe2006twisted}.

\begin{proposition}
\label{cor:twisted_graded}For $g\in G$, with notation as in \Cref{def:multiset},
\begin{equation}\lim_{x\to 1} \ |N|\cdot\P^{Ng}_\sigma\Bigl(x,x^t,qt(1-x)-1\Bigr) = t^{\fix_E (g)}\sum_{n \in N}\left(\prod\limits_{\bar{\lambda}_i(ng)\neq 1} \frac{1-\bar{\lambda}_i(ng)^\sigma}{1-\bar{\lambda}_i(ng)\vphantom{\Big|}} \right) q^{\fix_V (ng)}.\label{eq:subtle}\end{equation}
\end{proposition}

\begin{proof} Let us agree to index the pairs $(\epsilon_i^g(\VN),d_i^N)$ and $(\epsilon_i^g(\UNs),e_i^N(\Vs))$ in the multisets from \Cref{def:multiset} such that $\epsilon_i^g(\VN)=1$ for $1\leq i \leq \fix_{\VN}(g)$ (if $\fix_{\VN}(g)\neq 0$) and $\epsilon_i^g(\UNs)=1$ for $1\leq i\leq\fix_{\UNs}(g)$ (if $\fix_{\UNs}(g)\neq 0$). By \Cref{prop:twisted_graded-sigma}, the left-hand side of \Cref{eq:subtle} is \[|N|\cdot\lim_{x\to 1}\ \displaystyle\prod_{i=1}^r \left(\frac{\epsilon_i^g(\UNs) qt(1-x) x^{te_i^N(\Vs)}}{1-\epsilon_i^g(\VN) x^{d_i^N}}+\frac{1-\epsilon_i^g(\UNs)x^{te_i^N(\Vs)}}{1-\epsilon_i^g(\VN)x^{d_i^N}}\right).\] By \cite[Proposition 3.2]{bonnafe2006twisted}, $\fix_{\UNs} (g) \geq \fix_{\VN} (g)$. We will compute the above limit for the partial products ranging over $1\leq i \leq\fix_{\VN}(g)$; $\fix_{\VN}(g)+1\leq i\leq \fix_{\UNs}(g)$; and $\fix_{\UNs}(g)+1\leq i \leq r$ separately.

Since $\epsilon_i^g(\VN)=1=\epsilon_i^g(\UNs)$ for $1\leq i \leq \fix_{\VN}(g)$, we have that \[\lim_{x\to 1}\prod_{i=1}^{\fix_{\VN}(g)}\frac{qt(1-x)x^{te_i^N(\Vs)}}{1-x^{d_i^N}}+\frac{1-x^{te_i^N(\Vs)}}{1-x^{d_i^N}}=\prod_{i=1}^{\fix_{\VN}(g)}\frac{qt+te_i^N(\Vs)}{d_i^N}.\] Since $\epsilon_i^g(\VN)\neq 1\neq\epsilon_i^g(\UNs)$ for $\fix_{\UNs}(g)+1\leq i\leq r$, we have that \[\lim_{x\to 1}\prod_{i=\fix_{\UNs}(g)+1}^r\left(\frac{1+\epsilon_i^g(\UNs) \bigl(qt(1-x) -1\bigr)x^{te_i^N(\Vs)}}{1-\epsilon_i^g(\VN) x^{d_i^N}}\right)=\prod_{i=\fix_{\UNs}(g)+1}^r\frac{1-\epsilon_i^g(\UNs)}{1-\epsilon_i^g(\VN)}.\] If the inequality $\fix_{\UNs}(g)>\fix_{\VN}(g)$ is strict, so that $\epsilon_i^g(\VN)\neq 1=\epsilon_i^g(\UNs)$ for $\fix_{\VN}(g)+1\leq i\leq \fix_{\UNs}(g)$, then we see that for each such $i$ the limit of the corresponding factor is \[\lim_{x\to 1}\frac{1+\bigl(qt(1-x)-1\bigr)x^{te_i^N(\Vs)}}{1-\epsilon_i^g(\VN)x^{d_i^N}}=0.\] 

Therefore, if $\fix_{\UNs}(g)>\fix_{\VN}(g)$, then (cf.~\cite[Theorem~3.3]{bonnafe2006twisted})\[\lim_{x\to 1}|N|\cdot\P^{Ng}_\sigma\Bigl(x,x^t,qt(1-x)-1\Bigr)=0.\] On the other hand, if $\fix_{\UNs}(g)=\fix_{\VN}(g)$, then (cf.~\cite[Theorem~3.3]{bonnafe2006twisted})\[\lim_{x\to 1}|N|\cdot\P^{Ng}_\sigma\Bigl(x,x^t,qt(1-x)-1\Bigr)=t^{\fix_{\VN}(g)}\!\prod_{i=1}^{\fix_{\VN}(g)}\!(q+e_i^N(\Vs))\!\!\prod_{i=\fix_{\VN}(g)+1}^r\!\!\frac{1-\epsilon_i^g(\UNs)}{1-\epsilon_i^g(\VN)}d_i^N.\]

In any case, we have shown that the left-hand side of \Cref{eq:subtle} is $t^{\fix_{\VN}(g)}\cdot P(q)$ for some $P(q)\in\CC[q]$. To conclude the proof, it suffices to compare the left- and right-hand sides of \Cref{eq:subtle} at $t=1$. For this, we observe as in \cite[Theorem~3.3]{bonnafe2006twisted} that, as a consequence of \Cref{thm:twisted_os} and the arguments of \cite[Theorem~3.3]{orlik1980unitary} that are now standard, \begin{align*}\lim_{x\to 1} \ |N|\cdot\P^{Ng}_\sigma\Bigl(x,x,q(1-x)-1\Bigr)=&\sum_{n\in N}\left(\prod_{i=1}^r\frac{1+\bar{\lambda}_i(ng)^\sigma u}{1-\bar{\lambda}_i(ng)x\vphantom{\Big|}}\right)\Bigg|_{\substack{u=q(1-x)-1\\ x \to 1}}\\
=&\sum_{n \in N}\left(\prod\limits_{\bar{\lambda}_i(ng)\neq 1} \frac{1-\bar{\lambda}_i(ng)^\sigma}{1-\bar{\lambda}_i(ng)\vphantom{\Big|}} \right) q^{\fix_V (ng)}.\qedhere\end{align*}\end{proof}

\begin{corollary}[Sum side specialization]\label{cor:sum_side_specialization}
\[\lim_{x\to 1}|G|\cdot\P^G_\sigma\Bigl(x,x^t,qt(1-x)-1\Bigr)=\sum_{g \in G}\left( \prod_{\lambda_i(g) \neq 1} \frac{1-\lambda_i(g)^\sigma}{1-\lambda_i(g)}\right) q^{\fix_V (g)} t^{\fix_{E} (g)}.\]
\end{corollary}

\begin{proof}
Let $g_1,\dots,g_{|H|}\in G$ be a full set of coset representatives for $G/N=H$. Since $\P^G_\sigma(x,y,u)=\frac{1}{|H|}\sum_{j=1}^{|H|}\P^{Ng_j}_\sigma(x,y,u)$ and $\fix_{\VN}(g_j)=\fix_{\VN}(ng_j)$ for any $n\in N$, it follows from \Cref{cor:twisted_graded} that \begin{gather*}\lim_{x\to 1}\ |G|\cdot\P^G_\sigma\Bigl(x,x^t,qt(1-x)-1\Bigr)=\sum_{j=1}^{|H|}\lim_{x\to 1}|N|\cdot\P^{Ng_j}_\sigma\Bigl(x,x^t,qt(1-x)-1\Bigr)=\\
\begin{aligned}
=&\sum_{j=1}^{|H|}t^{\fix_{\VN}(g_j)}\sum_{n\in N}\left(\prod_{\bar{\lambda}_i(ng_j)\neq 1}\frac{1-\bar{\lambda}_i(ng_j)^\sigma}{1-\bar{\lambda}_i(ng_j)\vphantom{\Big|}}\right)q^{\fix_V(ng_j)}\\
=&\sum_{g\in G}\left(\prod_{\bar{\lambda}_i(g)\neq 1}\frac{1-\bar{\lambda}_i(g)^\sigma}{1-\bar{\lambda}_i(g)\vphantom{\Big|}}\right)q^{\fix_V(g)}t^{\fix_{\VN}(g)},\end{aligned}\end{gather*} and our result follows after replacing $g$ with $g^{-1}$.\end{proof}

\begin{remark}\label{rem:main_theorem_specializations}
As mentioned in the introduction, the formula of \Cref{thm:main_theorem} corresponding to the special case $\sigma=1$ becomes \cite[Theorem~1.5]{arreche2020normal}:
\begin{equation}\label{eq:main_theorem_2}\sum_{g\in G} q^{\fix_V(g)}t^{\fix_{\VN}(g)}=\prod_{i=1}^r(qt+e_i^N(V)t+e_i^G(\VN)),\end{equation}
which recovers \Cref{eq:shephard_todd} for the reflection group $G$ by evaluating at $t=1$, since $\VN\simeq\UN$ as $G$-modules in this case by \Cref{EU} and $e_i^N(V)+e_i^G(E)=e_i^G(V)$ by \Cref{thm:numbers}, as discussed in \Cref{rem:untwisted_numbers}.

On the other hand, specializing \Cref{eq:main_theorem_2} at $q=1$ and dividing by $|N|$ on both sides again recovers \Cref{eq:shephard_todd}, but this time for the reflection group $H$: the sum-side follows from observing that $\fix_{\VN}(Ng)=\fix_{\VN}(g)$ for every $Ng\in H$. The product side follows from the equality $d_i^N\cdot e_i^H(\VN)=e_i^G(\VN)$ proved in \Cref{thm:numbers}, which is compatible with the classical identities $1+e_i^N(V)=d_i^N$ and $\prod_{i=1}^rd_i^N=|N|$ by \Cref{rem:untwisted_numbers}.

In fact, it is also possible to recover \Cref{eq:shephard_todd} for the reflection group $N$ from \Cref{eq:main_theorem_2}. Since $H$ acts faithfully on $\VN$, we have $N=\{g\in G \ | \ \fix_{\VN}(g)=r\}$, and therefore applying $\frac{1}{r!}\frac{\partial^r}{\partial t^r}$ to the sum-side of \Cref{eq:main_theorem_2} recovers the sum-side of \Cref{eq:shephard_todd} for $N$. That the analogous result obtains for the product side follows from the well-known higher Leibniz rule for the Hasse-Schmidt derivations $\delta^{(i)}:=\frac{1}{i!}\frac{\partial^i}{\partial t^i}$, which yield \[\delta^{(r)}\left(\prod_{i=1}^r(qt+e_i^N(V)t+e_i^G(\VN))\right)=\prod_{i=1}^r\delta^{(1)}(qt+e_i^N(V)t+e_i^G(\VN))=\prod_{i=1}^r(q+e_i^N(V)).\]

Similarly, as we mentioned in the introduction, for arbitrary $\sigma\in\mathrm{Gal}(\mathbb{Q}(\zeta_G)/\mathbb{Q})$ the formula of \Cref{thm:main_theorem}
\begin{equation}\label{eq:main_theorem_1}
\sum_{g\in G}\left(\prod_{\lambda_i(g)\neq 1}\frac{1-\lambda_i(g)^\sigma}{1-\lambda_i(g)}\right)q^{\fix_V(g)}t^{\fix_{\VN}(g)}=\prod_{i=1}^r(qt+e_i^N(\Vs)t+e_i^G(\UNs))
\end{equation} recovers \Cref{thm:orlik_solomon} by evaluating at $t=1$, since $e_i^N(\Vs)+e_i^G(\UNs)=e_i^G(\Vs)$ by \Cref{thm:numbers}. The above arguments also show that we recover \Cref{thm:orlik_solomon} for the reflection group $N$ again by applying $\frac{1}{r!}\frac{\partial^r}{\partial t^r}$ to both sides of \Cref{eq:main_theorem_1}.

It would be desirable to recover \Cref{thm:orlik_solomon} for the reflection group $H$ from \Cref{thm:main_theorem} in analogy with the case $\sigma=1$, by evaluating \Cref{eq:main_theorem_1} at $q=1$ and dividing by $|N|$ on both sides. But in this case we obtain something else: letting $g_1,\dots,g_{|H|}\in G$ be a full set of coset representatives for $H=G/N$, evaluating \Cref{eq:main_theorem_1} at $q=1$ yields \[\sum_{j=1}^{|H|}\left(\sum_{n\in N}\left(\prod_{\lambda_i(ng_j)\neq 1}\frac{1-\lambda_i(ng_j)^\sigma}{1-\lambda_i(ng_j)}\right)\right)t^{\fix_{\VN}(g_j)}=\prod_{i=1}^r((e_i^N(\Vs)+1)t+e_i^G(\UNs)),\] which does not immediately compare to the statement of \Cref{thm:orlik_solomon} for the reflection group $H$: \[ \sum_{j=1}^{|H|}\left(\prod_{\bar{\epsilon}_i^{\, g_j}(\VN)\neq 1}\frac{1-\bar{\epsilon}_i^{\, g_j}(\VN)^\sigma}{1-\bar{\epsilon}_i^{\, g_j}(\VN)\vphantom{\Big|}}\right)t^{\fix_{\VN}(g_j)}=\prod_{i=1}^r(t+e_i^H(\VNs)),\] where compatibly with \Cref{def:multiset} the $\bar{\epsilon}_i^{\, g_j}(\VN)$ denote the eigenvalues of $g_j\in G$ acting on $\VN$.
\end{remark}

\section{Reflexponents revisited}
\label{sec:reflexponents}
Fix $G$ a complex reflection group of rank $r$ with reflection representation $V$.  Call an $r$-dimensional representation $M$ of $G$ \defn{factorizing} if $M$ has dimension $r$ and
\[\sum_{g \in G} q^{\mathrm{fix}_V(g)} t^{\mathrm{fix}_M(g)} = \prod_{i=1}^r \Big(qt+(e_i^G(V)-m_i)t+ m_i\Big),\] for some nonnegative integers $m_1,\ldots,m_r$.  More generally, call a representation $M$ of $G$ of dimension $\dim M \leq r$ factorizing if it is factorizing in the above sense after adding in $r-\dim M$ copies of the trivial representation.

A case-by-case construction of a factorizing representation $M_\mathcal{H}$ associated to an arbitrary orbit of reflecting hyperplanes $\mathcal{H}$ was presented in~\cite{williams2019reflexponents}---with two unexplained exceptions. These factorizing representations further restricted to the reflection representation of a parabolic subgroup supported on $\mathcal{H}$. We can now give a uniform explanation for those ad-hoc identities, including the two exceptions left unexplained in~\cite[Section~5.1]{williams2019reflexponents}.

Let $\mathcal{H}$ be an orbit of hyperplanes, write $\mathcal{R}_\mathcal{H}$ for the set of reflections fixing some $H \in \mathcal{H}$, and let $N_\mathcal{H} = \left\langle \mathcal{R}_\mathcal{H} \right\rangle$ be the subgroup generated by reflections around hyperplanes in $\mathcal{H}$.  Since these reflections form a conjugacy class in $G$, $N_\mathcal{H}$ is a normal reflection subgroup of $G$.  Furthermore:
\begin{enumerate}[(i)]
\item the quotient $G/N_\mathcal{H}$ acts by reflections on the vector space of $N_\mathcal{H}$-orbits $M_\mathcal{H}$;
\item this $G$-representation $M_\mathcal{H}$ is factorizing by~\Cref{thm:main_theorem}; and
\item the mysterious indexing of the reflexponents (i.e., the $e_i^G(M_\mathcal{H})$) left unexplained in~\cite{williams2019reflexponents} is now explained in~\Cref{rem:indexing}.
\end{enumerate}

\begin{corollary}
For $\mathcal{H}$ an orbit of hyperplanes and $N_\mathcal{H} = \left\langle \mathcal{R}_\mathcal{H} \right\rangle,$ the representation $M_\mathcal{H}$ is factorizing.
\end{corollary}

We can also explain the two exceptional factorizing representations from~\cite{williams2019reflexponents}:
\begin{itemize}
\item Following the conventions of~\cite{micheltable}, $G=G_{13}=\langle s,t,u\rangle$ was observed in~\cite[Section~5.1]{williams2019reflexponents} to have a two-dimensional representation with the factorizing property. The group $N = \langle gsg^{-1}: g \in G \rangle$ (fixing the conjugacy class $\mathcal{H}_s$) is a normal subgroup isomorphic to $G(4,2,2)$ and the quotient $G/N \simeq W(A_2)\simeq \mathfrak{S}_3$ gives the unexplained two-dimensional factorizing representation in this case.
\item For $G=G(ab,b,r)=\langle s,t_2,t_2',t_3,\ldots,t_r\rangle$ with $a,b>1$ and $r>2$, we can take $N= \langle gsg^{-1}: g \in G \rangle$.  $N$ is a normal subgroup of $G$ isomorphic to $(C_a)^r$ (it consists of diagonal matrices whose diagonal entries are $a$-th roots of unity).  The quotient $G/N \simeq G(b,b,r)$ gives the unexplained $r$-dimensional factorizing representation in~\cite[Section~5.2]{williams2019reflexponents}.
\end{itemize}

\section{Classification of Normal Reflection Subgroups}
\label{sec:classification}


In this section, we state the classification of normal reflection subgroups of irreducible complex reflection groups.

Recall that $G(ab,b,r)$ is given in its \defn{standard reflection representation} as the set of $r \times r$ monomial matrices whose every non-zero entry is an $(ab)$-th root of unity and in which the product of the non-zero entries is an $a$-th root of unity. The following theorem identifies the normal reflection subgroups of the infinite family $G(ab,b,r)$.

\begin{theorem}[{\cite[Chapter 2]{lehrer2009unitary}}]
Fix positive integers $a$ and $b$, and let $a=de$.

\noindent
For rank $r=1$, $G(ab,b,1)=C_a$ has normal subgroups and quotients
\begin{enumerate}
  \item[{\rm (1.a)}] $C_a \triangleright C_d \simeq C_e$.
\end{enumerate}
For rank $r=2$, the normal subgroups and quotients are:
\begin{enumerate}
\item[{\rm (2.a)}] $G(ab,b,2)/(C_d)^2 \simeq G(eb,b,2)$,
\item[{\rm (2.b)}] $G(ab,b,2)/G(ab,db,2) \simeq C_d$,
\item[{\rm (2.c)}] $G(2a,2,2)/G(a,d,2)\simeq C_2\times C_d$.
\end{enumerate}
For rank $r\geq 3$, the normal subgroups and quotients are:
\begin{enumerate}
\item[{\rm ($r$.a)}] $G(ab,b,r)/(C_d)^r \simeq G(eb,b,r)$ and
\item[{\rm ($r$.b)}] $G(ab,b,r)/G(ab,db,r) \simeq C_d$.
\end{enumerate}
\label{thm:classification}
\end{theorem}

In cases ($r$.a) for $r\geq 1$, the polycyclic group $N=(C_d)^r$ is included in $G$ as diagonal matrices with each non-zero entry a $d$-th root of unity. In cases ($r$.b) for $r\geq 2$, the normal reflection subgroup $N=G(ab,db,r)$ is included in $G$ via its standard reflection representation in $\CC^r$. In case ($2$.c), the normal reflection subgroup $N=G(a,d,2)$ occurs twice in $G$: once via its standard reflection representation, and once as the group generated by the reflections $\{\mathrm{diag}(\zeta_e^k,1),\ \mathrm{diag}(1,\zeta_e^k) \ | \ k=1,\dots,e-1\}$  (when $e\neq 1$) along with the reflections \[\left\{\begin{pmatrix}0 & \zeta_{2a}^{2k-1} \\ \zeta_{2a}^{-2k+1} & 0\end{pmatrix} \ \middle| \ k=0,\dots,a-1 \right\}.\] Independently of the factorization $a=de$, these two copies of $N$ are conjugate in $\mathrm{GL}_2(\mathbb{C})$ by the matrix $\mathrm{diag}(1,\zeta_{2a})$, which normalizes $G$.

Besides the subtlety in case ($2$.c) discussed above, there are other coincidences in rank $r=2$ where $G$ contains several copies of ``the same'' normal reflection subgroup $N$, related to the well-known isomorphisms $G(4,4,2)\simeq G(2,1,2)$ and $G(2,2,2)\simeq C_2\times C_2$ (see~\cite[Example 2.11]{lehrer2009unitary}). For example, $G(4,2,2)$ contains the dihedral group of order eight as a normal reflection subgroup in \emph{three} different ways:  $G(4,4,2)$ in one way as case ($2$.b), and $G(2,1,2)$ in two ways as case ($2$.c).  Similarly, $G(4,2,2)$ contains $C_2 \times C_2$ in three ways: $(C_2)^2$ in one way as case ($2$.a), and $G(2,2,2)$ in two ways as case ($2$.c).

The exceptional (that is, primitive) irreducible complex reflection groups $G$ and their normal reflection subgroups $N$ are listed in \Cref{tab:rank2}. This classification was computed with Sage~\cite{sagemath} using the code available at~\cite{nathancode}.  Most examples occur in rank $r=2$. In rank $r\geq 3$, every reflection has order $2$ or $3$ \cite[Theorem~8.4]{lehrer2009unitary}, and the only exceptional groups with more than a single orbit of reflections are $G_{26}$ and $G_{28}$ \cite[Table~D.2]{lehrer2009unitary}, which leads to the four non-trivial exceptional examples listed in \Cref{tab:rank2} in rank $r\geq 3$.

An isomorphism type of $N$ is unique up to conjugation in $\mathrm{GL}_r(\CC)$ sending isomorphic normal reflection subgroups to each other while stabilizing the reflection representation of $G$.  In fact, there are only two instances in the exceptional groups ($G_5 \triangleright G_4$ and $G_{28} \triangleright G(2,2,4)$) where an isomorphism type appears as a normal subgroup more than once. We gather the situations where the same isomorphism type of $N$ occurs more than once as a normal reflection subgroup of $G$ in the following result, where $N_{(i)}$ denote the different isomorphic copies of the same normal reflection subgroup $N\triangleleft G$.

\begin{table}[htbp]
\[\begin{array}{|r||c|} \hline
G & \text{List of pairs } \begin{array}{rl} N & \text{a normal reflection subgroup of }G \\ H & \text{the quotient } G/N \end{array}  \\ \hline \hline
G_5 & \begin{array}{c|c}G_4 & G_4 \\ C_3  & C_3\end{array} \\ \hline
G_6 & \begin{array}{c|c} G_4 & G(4,2,2) \\ C_2&C_3 \end{array}\\ \hline
G_7 & \begin{array}{c|c|c}G_5 & G_6 & G(4,2,2) \\ C_2 & C_3 & C_3 \times C_3 \end{array} \\ \hline \hline
G_8 & \begin{array}{c} G(4,2,2) \\ G(1,1,3) \end{array} \\ \hline
G_9 & \begin{array}{c|c|c|c} G_8 & G_{12} & G_{13} & G(4,2,2) \\ C_2 &  C_4 & C_2 & D_6 \end{array} \\ \hline
G_{10} & \begin{array}{c|c|c|c} G_5 & G_7 & G_8 & G(4,2,2) \\ C_4 & C_2 & C_3 & G(3,1,2) \end{array} \\ \hline
G_{11} & \begin{array}{c|c|c|c|c|c|c|c|c|c} G_5 & G_7 &G_8& G_9 & G_{10} & G_{12} & G_{13} & G_{14} & G_{15} & G(4,2,2) \\  C_2 \times C_4 & C_2 \times C_2 & C_2\times C_3 &C_3 & C_2 &  C_3 \times C_4 & C_3 \times C_2 & C_4 & C_2 & G(6,2,2) \end{array} \\ \hline
G_{13} & \begin{array}{c|c} G_{12} & G(4,2,2) \\ C_2 & G(1,1,3) \end{array} \\ \hline
G_{14} & \begin{array}{c|c} G_5 & G_{12} \\ C_2 & C_3 \end{array} \\ \hline
G_{15} & \begin{array}{c|c|c|c|c|c} G_5 & G_7 & G_{12} & G_{13} & G_{14} & G(4,2,2) \\ C_2 \times C_2 & C_2 & C_2 \times C_3 &C_3 & C_2 & G(3,1,2) \end{array} \\ \hline \hline
G_{17} & \begin{array}{c|c} G_{16} & G_{22} \\ C_2 & C_5 \end{array} \\ \hline
G_{18} & \begin{array}{c|c} G_{16} & G_{20} \\ C_3 & C_5 \end{array} \\ \hline
G_{19} & \begin{array}{c|c|c|c|c|c} G_{16}& G_{17} & G_{18} & G_{20} & G_{21} & G_{22} \\ C_2 \times C_3 & C_3 & C_2 & C_2 \times C_5 & C_5 & C_3 \times C_5 \end{array} \\ \hline
G_{21} & \begin{array}{c|c} G_{20} & G_{22} \\ C_2 & C_3 \end{array} \\ \hline \hline
G_{26} & \begin{array}{c|c} G(3,3,3) & G_{25} \\ G_4 & C_2 \end{array}\\ \hline
G_{28} & \begin{array}{c|c} G(2,2,4) & G(2,2,4) \\  G(1,1,3) & G(1,1,3) \end{array}\\ \hline
\end{array}\]
 \caption{The exceptional groups, their non-trivial normal reflection subgroups, and the corresponding quotient reflection groups.}
\label{tab:rank2}
\end{table}

\begin{theorem}\label{thm:repeated_normal} Suppose $G$ is an irreducible complex reflection group admitting normal reflection subgroups $N_{(1)},\dots,N_{(k)}$ for $k\geq 2$ that are pairwise isomorphic (as abstract groups) but not equal in $G$. Then $k\in\{2,3\}$ and the normalizer of $G$ in $\mathrm{GL}(V)$ permutes $\{N_{(1)},\dots,N_{(k)}\}$ transitively under conjugation. Moreover, precisely one of the following possibilities occurs.
\begin{enumerate}
\item $G=G(4,2,2)$ and there are three different normal reflection subgroups isomorphic to the dihedral group of order $8$:

\begin{enumerate} 
\item $N_{(1)}=G(2,1,2)$ as in \Cref{thm:classification}($2$.c) with $d=1$;
\item $N_{(2)}=\mathrm{diag}(1,-i)\cdot N_{(1)}\cdot \mathrm{diag}(1,i)$;
\item $N_{(3)}=G(4,4,2)$ as in \Cref{thm:classification}($2$.b) with $d=2$.
\end{enumerate}

\item $G=G(4,2,2)$ and there are $k=3$ different normal reflection subgroups isomorphic to the Klein $4$-group:
\begin{enumerate}
\item $N_{(1)}=G(2,2,2)$ as in \Cref{thm:classification}($2$.c) with $d=2$;
\item $N_{(2)}=\mathrm{diag}(1,-i)\cdot N_{(1)}\cdot \mathrm{diag}(1,i)$;
\item $N_{(3)}=(C_2)^2$ as in \Cref{thm:classification}($2$.a) with $d=2$.
\end{enumerate}
\item $G=G(2a,2,2)$ with $a>2$ and for each factor $d$ of $a$ there are $k=2$ different normal reflection subgroups isomorphic to $G(a,d,2)$: 
\begin{enumerate}
\item $N_{(1)}=G(a,d,2)$ as in \Cref{thm:classification}($2$.c); 
\item $N_{(2)}=\mathrm{diag}(1,\zeta_{2a}^{-1})\cdot N_{(1)}\cdot \mathrm{diag}(1,\zeta_{2a})$.
\end{enumerate}
\item $G=G_5$ and there are $k=2$ different normal reflection subgroups isomorphic to $G_4$.
\item $G=G_{28}$ and there are $k=2$ different normal reflection subgroups isomorphic to $G(2,2,4)$.
\end{enumerate}
\end{theorem}

\begin{proof} By \Cref{thm:classification} and \Cref{tab:rank2}, the possibilities above are exhaustive, and it is clear that they are mutually exclusive.

It remains to show that any two isomorphic normal reflection subgroups of $G$ are conjugate under an element of the normalizer of $G$. In every situation listed above, except for $G=G_{28}$, we can find a complex reflection group $W$ that contains $G$ as a normal reflection subgroup (in its standard reflection representation) but which does not contain (the isomorphism type of) $N$ as a normal reflection subgroup. Since there are at most three isomorphic copies of $N$ in each case, they must then form a single conjugacy class in $W$.

In cases (1) and (2), we can take $W$ to be any of $G_6$; $G_7$; $G_8$; $G_9$; $G_{10}$; $G_{11}$; $G_{13}$; or $G_{15}$; since all of these contain $G=G(4,2,2)$ as their only imprimitive normal reflection subgroup according to \Cref{tab:rank2}.

In case (3), we can take $W=G(2a,1,2)$, which contains $G=G(2a,2,2)$ as a normal reflection subgroup, but does not normalize $N=G(a,d,2)$ for any factor $d$ of $a$.

In case (4), we can take $W$ to be any of $G_7$; $G_{10}$; $G_{11}$; $G_{14}$; or $G_{15}$; since all of these contain $G_5$ as a normal reflection subgroup, but do not normalize $G_4$.

Finally, in case (5) it is not possible to find a complex reflection group $W$ containing $G=G_{28}$ as a normal reflection subgroup, since $G_{28}$ is the only irreducible complex reflection group admitting a non-trivial normal reflection subgroup in rank $r\geq 4$. To see that the two isomorphic copies of $N\simeq G(2,2,4)$ in $G=G_{28}$ are conjugate under an element of the normalizer of $G$ in $\mathrm{GL}_4(\mathbb{C})$, consider the set of reflecting hyperplanes for $G_{28}$, which are the orthogonal complements (with respect to the standard Hermitian inner product in $\CC^4$) of the lines in $\mathcal{L}_1\cup\mathcal{L}_2\cup\mathcal{L}_3$ defined by (cf.~\cite[Section~7.6.2]{lehrer2009unitary}): \begin{gather*}\mathcal{L}_1=\left\{\CC\cdot e_i \ \middle| \ 1\leq i\leq 4\right\};\qquad \mathcal{L}_2=\left\{\CC\cdot\tfrac{1}{2}(e_1 \pm e_2 \pm e_3\pm e_4)\right\}; \quad\text{and} \\ \mathcal{L}_3=\left\{\CC\cdot (e_i\pm e_j) \ \middle| \ 1\leq i<j\leq 4\right\};\end{gather*} where $e_i$ denotes the standard basis vector with $1$ in the $i$-th entry and $0$ elsewhere. There are two orbits of reflecting hyperplanes for $G_{28}=W(F_4)$ (the Weyl group of type $F_4$ and a real reflection group), corresponding to $\mathcal{L}_1\cup\mathcal{L}_2$ (the lines spanned by the short roots) and $\mathcal{L}_3$ (the lines spanned by the long roots). The reflections around the $12$ hyperplanes corresponding to $\mathcal{L}_3$ generate the normal reflection subgroup $N_{(1)}=G(2,2,4)$ acting in its standard reflection representation. The other normal reflection subgroup $N_{(2)}$ is generated by the reflections around the other $12$ hyperplanes corresponding to $\mathcal{L}_1\cup\mathcal{L}_2$. To conclude the proof, note that the real orthogonal matrix \[P:=\frac{1}{\sqrt{2}}\begin{pmatrix} 0 & 1 & 0 & 1\\ 1 & 0 & 1 & 0 \\ 0 & 1 & 0 & -1 \\ -1 & 0 & 1 & 0\end{pmatrix} \] exchanges the two $G_{28}$-orbits of reflecting hyperplanes $\mathcal{L}_1\cup\mathcal{L}_2 \leftrightarrow\mathcal{L}_3$ (identifying $\mathcal{L}_1$ with the set of lines $\CC\cdot(e_i\pm e_j)\in\mathcal{L}_3$ such that $j-i=2$). Hence, $P$ normalizes $G_{28}$ and exchanges $N_{(1)}$ and $N_{(2)}$ under conjugation. \end{proof}

\section{Examples}
\label{sec:examples}
In this section we illustrate our results with examples, beginning with the cyclic groups in \Cref{sec:cyclic_examples}, continuing with the infinite family in \Cref{sec:infinite_examples}, and concluding with a non--well-generated exceptional example in \Cref{sec:exceptional_example}.

\subsection{Cyclic Groups}\label{sec:cyclic_examples}

Consider $G=C_a=\langle c\rangle$, the cyclic group of order $a$, acting on $V=\mathbb{C}$ via $c\mapsto \zeta_a^{-1}$, where $\zeta_a$ is a primitive $a$-th root of unity. Verifying~\Cref{thm:orlik_solomon} and determining the Orlik-Solomon space $\UGs$ from~\Cref{def:OS_space} for $\sigma\in\mathrm{Gal}(\mathbb{Q}(\zeta_a)/\mathbb{Q})$ is already an interesting calculation in this case.

Let $\sigma \in \mathrm{Gal}(\mathbb{Q}(\zeta_a)/\mathbb{Q})$ act as $\sigma:\zeta_a \mapsto \zeta_a^s$, where $s\in\mathbb{N}$ is coprime to $a$. Although it is sufficient to only consider $s\in\{1,\dots,a-1\}$, it will be essential to allow more general positive exponents $s$ in the description of the action of $\sigma$ on different roots of unity when we begin considering (normal reflection) subgroups of $G$ shortly. We compute the identity of \Cref{thm:orlik_solomon} in this example:
\begin{align*}
\sum_{g \in G} \left(\prod_{\lambda_1(g) \neq 1} \frac{1-\lambda_1(g)^\sigma}{1-\lambda_1(g)}\right) q^{\fix_{V} (g)} &= q+\sum_{k=1}^{a-1} \frac{1-\zeta_a^{ks}}{1-\zeta_a^k} = q+\sum_{k=1}^{a-1} \sum_{j=0}^{s-1} \zeta_a^{kj}\\ &= q+ \left(\sum_{j=0}^{s-1}\sum_{k=0}^{a-1} \zeta_a^{kj}\right)-s \\& = q+ \left(\sum_{j=0}^{s-1}\begin{cases} a & \text{if } a\,|\,j \\ 0 & \text{ otherwise} \end{cases}\right)-s \\
&= q+a\left\lceil \frac{s}{a} \right\rceil-s.
\end{align*}

On the other hand, we verify that $e_i^G(\Vs) =a\left\lceil \frac{s}{a} \right\rceil - s$ by exhibiting $x^{a\left\lceil \frac{s}{a}\right\rceil-s} \otimes x^\sigma$ as a basis vector for the dual $\UGsd$ of the Orlik Solomon space $\UGs$, where $x$ and $x^\sigma$ denote basis vectors for $\Vd$ and $\Vsd\!\!$, respectively. More generally, we have the following.

\begin{lemma}\label{lem:cyclic_twisted_exponents}

For $G=C_a=\langle c\rangle$ the cyclic group of order $a$ acting on its reflection representation $V=\mathbb{C}$ by $c\mapsto \zeta_a^{-1}$, and for any $s\in\mathbb{Z}$ (not necessarily coprime to~$a$), the $V^{\otimes s}$-exponent of $G$ is $e_1^G(V^{\otimes  s})=a\left\lceil \frac{s}{a} \right\rceil - s$.
\end{lemma}

\begin{proof}
Note that $a\left\lceil \frac{s}{a} \right\rceil - s\in\{0,\dots,a-1\}$ is congruent to $-s$ $(\mod a)$. Letting $x^s$ denote a basis vector for $V^{\otimes s}$, we see that $c(x^s)=\zeta_a^sx^s$ and therefore $x^{a\left\lceil \frac{s}{a} \right\rceil - s}\otimes x^s$ is a basis vector for $(U_{V^{\otimes s}}^G)^*=(\CG\otimes (V^{\otimes s})^*)^G$, since $\CG=\spn_\CC\{1,x,\dots,x^{a-1}\}$ in this case.
\end{proof}

Suppose now that $a=de$ and consider $N=\langle c^e\rangle\simeq C_d$, which is a normal reflection subgroup of $G$ with quotient $H=G/N \simeq C_e$. We denote by $\zeta_d:=\zeta_a^e$ and $\zeta_e:=\zeta_a^d$, so that $N$ acts on $V^*$ via $c^e\mapsto\zeta_d$ and $H=\langle cN\rangle$ acts on $\VNd=\spn_\CC\{x^d\}$ via $(cN)(x^d)=\zeta_e x^d$.

Taking again $s\in\mathbb{N}$ coprime to $a$ and letting $\sigma\in\mathrm{Gal}(\mathbb{Q}(\zeta_a)/\mathbb{Q})$ be given by $\sigma: \zeta_a \mapsto \zeta_a^s$, we have that $\VNs$ is the one-dimensional representation of $G$ defined by $c\mapsto \zeta_a^{-ds}$, or equivalently the one-dimensional representation of $H$ defined by $cN\mapsto \zeta_e^{-s}$. Therefore, by \Cref{lem:cyclic_twisted_exponents}, $e_1^G(\VNs)=a\left\lceil\frac{ds}{a}\right\rceil-ds$ and $e_1^H(\VNs)=e\left\lceil\frac{s}{e}\right\rceil-s$. On the other hand, $(\UNs)^* = \spn_\mathbb{C}\{ x^{d\left\lceil \frac{s}{d}\right\rceil-s} \otimes x^\sigma\}$ by \Cref{lem:cyclic_twisted_exponents}, where again $x$ and $x^\sigma$ denote basis vectors for $\Vd$ and $\Vsd$, respectively. Hence the generator $c$ of $G$ that acts by $\zeta_a^{-1}$ on its reflection representation $V$ now acts by $\zeta_{a}^{-d\left\lceil \frac{s}{d}\right\rceil}$ on $\UNs$, and therefore another application of \Cref{lem:cyclic_twisted_exponents} yields\[e_1^G(\UNs)=a\left\lceil\frac{\left\lceil\frac{s}{d}\right\rceil}{e}\right\rceil-d\left\lceil \frac{s}{d}\right\rceil.\] 

When $s=1$, we see that $\zeta_{a}^{-d\left\lceil \frac{s}{d}\right\rceil} = \zeta_e^{-1}$, so that $\UN \simeq \VN$ is the reflection representation of $H=C_e$ defined by $cN\mapsto \zeta_e^{-1}$. But since $\left\lceil \frac{s}{d}\right\rceil$ is not necessarily coprime to $e$, even when $s\in\{1,\dots,a-1\}$, it is possible for the action of $H$ on $\UNs$ to fail to be faithful---for example, when $s=a-1$ corresponding to $\sigma$ acting by complex conjugation, the generator $c$ of $G$ acts trivially on $\UNsd$ by $\zeta_a^{d\left\lceil \frac{a-1}{d}\right\rceil}=\zeta_e^{e}=1$. Hence we see that, as we mentioned in \Cref{rem:E_aint_UNs} and contrary to what one might have hoped based on the $s=1$ case, in general $\UNs \not\simeq \VN^\sigma$ as $G$-representations.

Using the explicit descriptions of the actions of $N$ and $G$ on $\Vs$, and the actions of $G$ and $H$ on $\UNs$ and $\VNs$ described above, the three equalities in~\Cref{thm:numbers} become the following numerological statements.

\begin{corollary}\label{cor:cyclic_numberology}Let $a=de$ and $s\in\mathbb{N}$ be coprime to $a$. Then
\begin{align*}
\left(d\left\lceil \frac{s}{d}\right\rceil -s\right) +\left(a\left\lceil\frac{\left\lceil\frac{s}{d}\right\rceil}{e}\right\rceil-d\left\lceil \frac{s}{d}\right\rceil\right) &= \left(a \left\lceil \frac{s}{a}\right\rceil - s\right);\\
d\cdot \left(e\left\lceil\frac{s}{e}\right\rceil-s\right)  &= \left(a\left\lceil\frac{ds}{a}\right\rceil-ds\right);\\
d\cdot e &= a.
\end{align*}
\end{corollary}
The first (and only non-trivial) equality in this case is equivalent to the identity $\left\lceil\frac{\left\lceil\frac{s}{d}\right\rceil}{e}\right\rceil=\left\lceil \frac{s}{de}\right\rceil$, which holds more generally for $e \in \NN$ and $s,d \in \RR$.

Our \Cref{thm:main_theorem} in this situation states that the following expressions are equal.
\begin{align*}\sum_{g \in G} \left(\prod_{\lambda_1(g)\neq 1}\!\!\!\frac{1-\lambda_1(g)^s}{1-\lambda_1(g)}\right)\! q^{\fix_V (g)}t^{\fix_E (g)} &= qt+\left(\sum_{j=1}^{d-1}\frac{1-\zeta_d^{js}}{1-\zeta_d^{j}}\right)\!t+\left(\sum_{\substack{ k=1\\ e\, \nmid \, k}}^{a-1} \frac{1-\zeta_a^{ks}}{1-\zeta_a^k}\right)\\
\left(qt+e_1^N(\Vs)t+e_1^G(\UNs)\right)&= qt+\left(d \left\lceil\frac{s}{d}\right\rceil -s\right)t+\left(a\left\lceil\frac{\left\lceil\frac{s}{d}\right\rceil}{e}\right\rceil-d\left\lceil \frac{s}{d}\right\rceil\right)
.\end{align*}
The equality of the coefficients of $t$ \[\sum_{j=1}^{d-1}\frac{1-\zeta_d^{js}}{1-\zeta_d^{j}}=d\left\lceil\frac{s}{d}\right\rceil-s(=e_1^N(\Vs))\] was already verified above with $a$ in place of $d$. To verify the equality of constant terms \[\sum_{\substack{ k=1\\ e\, \nmid \, k}}^{a-1} \frac{1-\zeta_a^{ks}}{1-\zeta_a^k}=a\left\lceil\frac{\left\lceil\frac{s}{d}\right\rceil}{e}\right\rceil-d\left\lceil \frac{s}{d}\right\rceil,\] one could proceed for example by noting that the same arguments show that the left-hand side is equal to \[\left(\sum_{k=1}^{a-1}\frac{1-\zeta_a^{ks}}{1-\zeta_a^{k}}\right)-\left(\sum_{j=1}^{d-1}\frac{1-\zeta_d^{js}}{1-\zeta_d^{j}}\right)=\left(a\left\lceil\frac{s}{a}\right\rceil-s \right)-\left(d\left\lceil\frac{s}{d}\right\rceil-s\right)=a\left\lceil\frac{s}{a}\right\rceil -d\left\lceil\frac{s}{d}\right\rceil\] and then appealing either to \Cref{cor:cyclic_numberology} or to more general properties of ceilings to obtain the equality $\left\lceil\frac{\left\lceil\frac{s}{d}\right\rceil}{e}\right\rceil=\left\lceil \frac{s}{a}\right\rceil$.

We conclude our discussion of the cyclic case with a concrete illustration of the subtlety involved in proving (in \Cref{cor:sum_side_specialization}) that the sum side of \Cref{thm:main_theorem} provides the correct contribution coset-by-coset, but not term-by-term as in the proof of \Cref{thm:orlik_solomon}.

To compare the two situations, we compute the trace of $ng=c^{ek}\cdot c^j\in Ng$ (for $0 \leq k < d$ and $0 \leq j < e$) acting on $S(V^*) \otimes \bigwedge (V^\sigma)^*$ and $S(V^*) \otimes \bigwedge (\UNs)^*$:
\begin{align*}
\sum_{\ell,p\geq 0}\mathrm{tr}\left((ng) | S(V^*)_\ell \otimes {\textstyle \bigwedge^p}(V^\sigma)^*\right)x^\ell u^p &=\frac{1+u (c^{ek+j})|_{(V^\sigma)^*}}{1-x(c^{ek+j})|_{(V)^*}}\Bigg|_{\substack{u=q(1-x)-1\\ x \to 1}}\\
&= \frac{1+u \zeta_a^{s(ek+j)}}{1-x\zeta_a^{ek+j}}\Bigg|_{\substack{u=q(1-x)-1\\ x \to 1}}\\
&=\frac{1- \zeta_a^{s(ek+j)}}{1-x\zeta_a^{ek+j}}+\frac{q(1-x) \zeta_a^{s(ek+j)}}{1-x\zeta_a^{ek+j}}\Bigg|_{x \to 1}\\
&=\left(\prod_{\lambda_1(ng)\neq 1} \frac{1-\lambda_1(ng)^\sigma}{1-\lambda_1(ng)} \right)q^{\fix_V(ng)},
\end{align*}
since the term $\frac{q(1-x) \zeta_a^{s(ek+j)}}{1-x\zeta_a^{ek+j}}$ vanishes in the limit $x \to 1$ for all elements of $G$ except the identity.  For example, for $a=6,d=2,e=3,$ and $s=5$, summing over all elements of $C_6=\langle c \rangle$ and then specializing $u\mapsto q(1-x)-1$ and taking the limit as $x \to 1$ gives
\begin{align*}
\begin{array}{cc|cc|cc}
\multicolumn{2}{c}{N}& \multicolumn{2}{c}{cN} & \multicolumn{2}{c}{c^2N}\\
\mathrm{id} & c^3 & c & c^4 & c^2 & c^5\\ \hline
\multicolumn{1}{c@{\hspace*{\tabcolsep}\makebox[0pt]{+}}}{q} & \multicolumn{1}{c@{\hspace*{\tabcolsep}\makebox[0pt]{+}}}{\frac{1-\zeta_6^3}{1-\zeta_6^3}} & \multicolumn{1}{c@{\hspace*{\tabcolsep}\makebox[0pt]{+}}}{\frac{1-\zeta_6^5}{1-\zeta_6}} &\multicolumn{1}{c@{\hspace*{\tabcolsep}\makebox[0pt]{+}}}{\frac{1-\zeta_6^2}{1-\zeta_6^4}} & \multicolumn{1}{c@{\hspace*{\tabcolsep}\makebox[0pt]{+}}}{\frac{1-\zeta_6^4}{1-\zeta_6^2}}  & \frac{1-\zeta_6^1}{1-\zeta_6^5}\\
\multicolumn{2}{c|}{q+1} & \multicolumn{2}{c|}{0} & \multicolumn{2}{c}{0}
\end{array}=q+1,
\end{align*}
so that each element contributes the ``correct amount'' specified by the sum side of~\Cref{thm:orlik_solomon}.

On the other hand, computing the sum side of our refinement~\Cref{thm:main_theorem} gives the following, where we have written $b:=e_1^N(\Vs)=d\left\lceil\frac{s}{d}\right\rceil -s$ for legibility.
\begin{align*}
\sum_{\ell,p\geq 0}\!\!\mathrm{tr}\left((ng) | S(V^*)_\ell \otimes \textstyle\bigwedge^p\UNsd\right)x^\ell y^{bp}u^p &=
\frac{1+u y^b (c^{ek+j})|_{\UNsd_{\!\!\!b}}}{1-x(c^{ek+j})|_{(V)^*}}\Bigg|_{\substack{y=x^t \\ u=qt(1-x)-1\\ x \to 1}}\\
&= \frac{1+ux^{tb} \zeta_e^{\left\lceil \frac{s}{d}\right\rceil(ek+j)}}{1-x\zeta_a^{ek+j}}\Bigg|_{\substack{u=qt(1-x)-1\\ x \to 1}}\\
&= \frac{1- x^{tb} \zeta_e^{j\left\lceil \frac{s}{d}\right\rceil}}{1-x\zeta_a^{ek+j}}+\frac{x^{tb} qt (1-x) \zeta_e^{j\left\lceil \frac{s}{d}\right\rceil}}{1-x\zeta_a^{ek+j}}\Bigg|_{x \to 1}\\
&=
\begin{cases}
qt+bt& \text{if } k=j=0\\
\frac{1- \zeta_e^{j\left\lceil \frac{s}{d}\right\rceil}}{1-\zeta_a^{ek+j}} & \text{ otherwise}.
\end{cases}
\end{align*}

\noindent Continuing the example with $a=6,d=2,e=3,$ and $s=5$, since $\lceil \frac{s}{d}\rceil=e=3$, after specializing $y\mapsto x^t$, $u\mapsto qt(1-x)-1$, and taking the limit as $x\to 1$, every term except for the identity vanishes---in particular, each element of $G$ \emph{does not} contribute the ``correct amount'' specified by the sum side of~\Cref{thm:main_theorem}:
\begin{align*}
\begin{array}{cc|cc|cc}
\multicolumn{2}{c}{N}& \multicolumn{2}{c}{cN} & \multicolumn{2}{c}{c^2N}\\
\mathrm{id} & c^3 & c & c^4 & c^2 & c^5\\ \hline
\multicolumn{1}{c@{\hspace*{\tabcolsep}\makebox[0pt]{+}}}{(qt+t)} & \multicolumn{1}{c@{\hspace*{\tabcolsep}\makebox[0pt]{+}}}{0} & \multicolumn{1}{c@{\hspace*{\tabcolsep}\makebox[0pt]{+}}}{0} &\multicolumn{1}{c@{\hspace*{\tabcolsep}\makebox[0pt]{+}}}{0} & \multicolumn{1}{c@{\hspace*{\tabcolsep}\makebox[0pt]{+}}}{0}  & 0\\
\multicolumn{2}{c|}{qt+t} & \multicolumn{2}{c|}{0} & \multicolumn{2}{c}{0}
\end{array}=qt+t.
\end{align*}
Here the only coset that provides a non-trivial contribution to the sum side of \Cref{thm:main_theorem} is the trivial coset $N$, as predicted by the computation of the sum side of \Cref{thm:orlik_solomon}, but this non-trivial contribution of $qt+t$ for the whole coset $N$ is concentrated on the identity element $c^0\in N$ alone, which according to the sum side of \Cref{thm:main_theorem} should have only contributed $qt$, whereas the non-trivial element $c^3\in N$ did not provide the correct contribution of $\frac{1-\lambda_1(c^3)^\sigma}{1-\lambda_1(c^3)}q^{\fix_V(c^3)}t^{\fix_\VN(c^3)}=t$ specified by the sum side of \Cref{thm:main_theorem}.

\subsection{The Infinite Family~\texorpdfstring{$G(ab,b,r)$}{G(ab,b,r)}}\label{sec:infinite_examples}

We begin by defining an ad-hoc operation on $G$-modules for $G=G(ab,b,r)$ (whose definition was given in \Cref{sec:classification}) that will allow us to succintly identify the Orlik-Solomon spaces $\UNs$ of \Cref{def:OS_space} for the two kinds of normal reflection subgroups $N$ of $G$ listed in \Cref{thm:classification} that occur in all ranks $r\geq 2$.

\begin{definition}\label{def:fake_tensor} For $n\in\ZZ$, let $\mu_n:G\rightarrow G$ be the group endomorphism obtained by raising each non-zero matrix entry to the $n$-th power. Given any representation $\rho:G\rightarrow\mathrm{GL}(W)$, we define the \defn{fake tensor power} $W^{\boxtimes n}$ as the representation $\rho\circ\mu_n:G\rightarrow\mathrm{GL}(W)$.
\end{definition}

\begin{remark}
Note that in general the ``fake $n$-th power map'' $\mu_n$ used in \Cref{def:fake_tensor} will be neither injective nor surjective when $\mathrm{gcd}(ab,n)\neq 1$. In the case where $\mathrm{gcd}(ab,s)=1$, note that $V^{\boxtimes s}\simeq V^\sigma$, the Galois twist corresponding to $\sigma:\zeta_{ab}\mapsto\zeta_{ab}^s$.

Although the fake tensor power operation could be expressed in terms of more systematic constructions (interpreting $G(ab,b,r)$ as an index-$b$ subgroup of the wreath product $C_{ab}\wr S_r$), we have preferred our ad hoc definition for its simplicity and concreteness.\end{remark}

In the following result, we identify the $G$-module $\UNs$ when $N=C_d^r$ as in \Cref{thm:classification}($r$.a) for $r\geq 2$.

\begin{proposition}\label{prop:infinite_normal_2}
Let $a=de$ and $N=C_d^r\triangleleft G(ab,b,r)=G$, and fix $\sigma: \zeta_{ab} \to \zeta_{ab}^s$ for $1 \leq s < ab$ with $\gcd(s,ab)=1$.   Then $H\simeq G(eb,b,r)$ and $\UNs\simeq\VN^{\boxtimes \left\lceil \frac{s}{d}\right\rceil }$ as $G$-modules.
\end{proposition}

\begin{proof} 
For this choice of $N$, the fundamental $N$-invariants are $\n_i = x_i^d$ for $1 \leq i \leq r$, and we obtain the basis $u_i^N= x_i^{d\left\lceil\frac{s}{d}\right\rceil-s} \otimes x_i^\sigma$ for $\UNsd$ as in \Cref{sec:cyclic_examples}. 

Fix $g\in G$ and $\ell\in\{1,\dots,r\}$, and suppose that $g(x_\ell)=\zeta_{ab}^k x_j$ for some $0\leq k<ab$ and $j\in\{1,\dots,r\}$. Then from the explicit descriptions of $\VNd$ and $\UNsd$ above we see that $g(\n_\ell)=\zeta_{ab}^{kd}\n_j$ and $g(u_{\ell}^N)=\zeta_{ab}^{kd\left\lceil\frac{s}{d}\right\rceil}u_j^N$. It follows that $\VN\simeq V^{\boxtimes d}$ and $\UNs\simeq V^{\boxtimes d\left\lceil\frac{s}{d}\right\rceil }$, and therefore $\UNs\simeq \VN^{\boxtimes \left\lceil\frac{s}{d}\right\rceil}$, as claimed.\end{proof}

As a special case of \Cref{prop:infinite_normal_2}, $\UNs\simeq \VN$ when $N=C_d^r$ and $s<d$.

\medskip
Following~\cite[Proposition~14.1]{reiner2019invariant}, a simple choice of invariant polynomials for $G(ab,b,r)$ is the set 

\begin{equation} \label{eq:infinite_invariants}
G_i=\sum_{j=1}^r x_j^{abi} \quad\text{for}\ 1 \leq i < r \qquad\text{and}\qquad G_r = (x_1 \cdots x_r)^a. \end{equation}

\noindent In the following result we determine the Orlik-Solomon space $\UGs$ of \Cref{def:OS_space} up to graded cryptomorphism.

\begin{theorem} \label{thm:infinite_exponents} Let $G=G(ab,b,r)$ and $\sigma \in \mathrm{Gal}(\mathbb{Q}(\zeta_{ab})/\mathbb{Q})$ be given by $\sigma(\zeta_{ab})=\zeta_{ab}^s$ for $1 \leq s < ab$ with $\gcd(s,ab)=1$. Let $\UGstd:=\spn_\CC\{\tilde{u}_1^G,\dots,\tilde{u}_r^G\}$, where
\[\tilde{u}_i^G = \sum_{j=1}^r x_j^{abi-s}\otimes x_j^\sigma \quad \text{for}\ \ 1 \leq i < r \qquad \text{and}\qquad \tilde{u}_r^G = \sum_{j=1}^r (x_1 \cdots x_r)^{\left\lceil \frac{s}{a}\right\rceil a} x_j^{-s} \otimes x_j^\sigma.\]  
Let $\eta_G:S(V^*)\rightarrow\CG$ denote the natural projection from $S(V^*)$ onto its $G$-stable direct summand $\CG$. Then $\eta_G\otimes 1:\UGstd\rightarrow\UGsd$ is an isomorphism of graded vector spaces, and therefore the $\Vs$-exponents of $G(ab,b,r)$ are \[\left\{ab-s,2ab-s,\ldots,(r{-}1)ab-s,\left\lceil \frac{s}{a}\right\rceil ar-s\right\}.\]
\end{theorem}

\begin{proof}
By \Cref{lem:fake_OS_space} (applied in the special case where $N=G$), it suffices to show that the $\tilde{u}_i^G$ form a basis for $(S(V^*)\otimes\Vsd)^G$ as a free $S(V^*)^G$-module.

Since the group $G(ab,b,r)$ acts on $V^*$ by permutation of the $x_i$ and multiplication by $(ab)$-th roots of unity, it is clear that the $\tilde{u}_i^G\in S(V^*)\otimes\Vsd$ are $G$-invariant. Let $u_i^G=\sum_{j=0}^ra_{ij}^G\otimes x_j^\sigma$ be a homogeneous basis for $\UGsd$, where $a_{ij}^G\in\CG$ and $\mathrm{deg}(a_{ij}^G)=e_i^G(\Vs)$ whenever $a_{ij}^G\neq 0$. Since $(S(V^*)\otimes\Vsd)^G\simeq S(V^*)^G\otimes\UGsd$, there exists a matrix $[p_{ij}^G]\in\mathrm{Mat}_{r\times r}(S(V^*)^G)$ such that $[p_{ij}^G]\cdot[a_{ij}^G]=[\tilde{a}_{ij}^G]$, where \[\tilde{a}_{ij}^G = \begin{cases} x_j^{abi-s} & \text{ for  } 1 \leq i <r \\ (x_1\cdots x_r)^{\left\lceil \frac{s}{a}\right\rceil a} x_j^{-s} & \text{ for  } i=r,\end{cases}\] or equivalently such that $\tilde{u}^G_i=\sum_{j=1}^rp_{ij}^G\cdot u_j^G$. It is clear from the form of the $\tilde{u}_i^G$ that they are $\CC$-linearly independent, which implies that $\mathrm{det}(p_{ij}^G)\neq 0$.

We claim that \begin{equation}\label{eq:gutkin}\mathrm{det}(a_{ij}^G)=c\cdot (x_1\cdots x_r)^{\left\lceil \frac{s}{a}\right\rceil a-s} \prod_{1\leq i < j \leq r} (x_i^{ab}-x_j^{ab})\end{equation} for some $0\neq c\in\CC$. By Gutkin's Theorem~\cite[Theorem 10.13]{lehrer2009unitary},\[\mathrm{det}(a_{ij}^G)=c\cdot \prod_{H\in\mathcal{R}_G} L_H^{C(H,\Vs)}\] for some $0\neq c\in\CC$, where $\mathcal{R}_G$ denotes the set of reflecting hyperplanes for $G$, $L_H\in V^*$ denotes a linear form defining $H$, and $C(H,\Vs)$ is defined as follows. Denoting by $\langle r_H\rangle = G_H<G$ the cyclic subgroup of $G$ that stabilizes $H$ pointwise, decompose $\Vsd\simeq \bigoplus_{i=1}^r\lambda^{\otimes k_i}$ as a $G_H$-module with $0\leq k_i<|G_H|$ for $1\leq i\leq r$, where $\lambda$ denotes the standard representation of $G_H$ on the $G_H$-stable complement of $H$ in $V$, and define $C(H,\Vs):=\sum_{i=1}^r k_i$. When $L_H=x_i-\zeta_{ab}^\ell x_j$ with $i\neq j$ and $0\leq\ell<ab-1$, the cyclic generator $r_H$ has order $2$ and $\Vs\simeq\lambda \oplus\CC^{\oplus(r-1)}$ as a $G_H$-module, and therefore $C(H,\Vs)=1$ in this case. If $a>1$, then we have the additional reflecting hyperplanes defined by $L_H=x_i$; in this case, the cyclic generator $r_H$ has order $a$ and $\Vsd\simeq \lambda^{\otimes\left(\left\lceil\frac{s}{a}\right\rceil a -s\right)}\oplus\CC^{\oplus(r-1)}$ as a $G_H$-module (cf.~\Cref{lem:cyclic_twisted_exponents}), and therefore $C(H,\Vs)=\left\lceil\frac{s}{a}\right\rceil a -s$. This concludes the proof of \Cref{eq:gutkin}.

We see by direct inspection that \[\mathrm{deg}(\mathrm{det}(a_{ij}^G))=ab\cdot\binom{r}{2}+r\cdot\left(\left\lceil\frac{s}{a}\right\rceil a-s\right)=\mathrm{deg}(\mathrm{det}(\tilde{a}_{ij}^G)),\] which implies that $\mathrm{deg}(\mathrm{det}(p_{ij}^G))=0$, and since $\mathrm{det}(p_{ij}^G)\neq 0$ as we had already seen, it follows that $[p_{ij}^G]\in\mathrm{GL}_r(S(V^*)^G)$, as we wanted to show.\end{proof}

\begin{remark}Note that when $s=ab-1$, so that $\sigma$ acts by complex conjugation, the basis $\tilde{u_i}^G$ for $\UGstd$ in \Cref{thm:infinite_exponents} agrees with the one computed in \cite[Section~6]{orlik1980unitary}.\end{remark}

In the following result, we identify the $G$-module $\UNs$ when $N=G(ab,db,r)$ as in \Cref{thm:classification}($r$.b) for $r\geq 2$.

\begin{proposition} \label{prop:infinite_normal_1}
Let $a=de$ and $N=G(ab,db,r)\trianglelefteq G(ab,b,r)=G$, and fix $\sigma: \zeta_{ab} \to \zeta_{ab}^s$ for $1 \leq s < ab$ with $\gcd(s,ab)=1$. Then $H\simeq C_d$ and $\UNs\simeq \VN^{\boxtimes \left\lceil \frac{s}{e}\right\rceil}$ as $G$-modules. 
\end{proposition}

\begin{proof}
Note that for this choice of $G$ and $N$, with fundamental invariants defined as in \Cref{eq:infinite_invariants}, we have $\n_i=\g_i$ for $i=1,\dots,r-1$, and $\n_r=(x_1\cdots x_r)^e$ while $\g_r=(x_1\cdots x_r)^a$. The matrix $\mathrm{diag}(\zeta_a,1,\dots,1)\in G$ maps to the cyclic generator $c\in C_d\simeq H$, acting trivially on $\spn_\CC\{\n_1,\dots,\n_{r-1}\}$ and mapping $c:\n_r\mapsto\zeta_d\n_r$, where $\zeta_d:=\zeta_a^{-e}$ is a primitive $d$-th root of unity.

On the other hand, we see that the space $\tilde{U}^N_\sigma=\spn\{\tilde{u}_1^N,\dots,\tilde{u}_r^N\}$ constructed in \Cref{thm:infinite_exponents} is $G$-stable and the $\tilde{u}_i^N$ form a homogeneous basis for $(S(V^*)\otimes\Vsd)^N$ as a free $S(V^*)^N$-module, and therefore it is isomorphic to $\UNsd$ as a graded $G$-module by \Cref{lem:fake_OS_space}. We see that $\mathrm{diag}(\zeta_a,1,\dots,1)\in G$ acts trivially on $\spn_\CC\{\tilde{u}_1^N,\dots,\tilde{u}_{r-1}^N\}$ and maps $c:\tilde{u}_r^N\mapsto \zeta_a^{-e\left\lceil\frac{s}{e}\right\rceil}\tilde{u}_r^N=\zeta_d^{\left\lceil\frac{s}{e}\right\rceil}\tilde{u}_r^N$. \end{proof}

Note that we do not necessarily have that $\lceil \frac{s}{e}\rceil$ is relatively prime to $d$. When $d=a$ so that $e=1$ and $N=G(ab,ab,r)$, $G/N\simeq C_a=\langle c\rangle$ acts by $c: u_r^N \to \zeta_a^{-s} u_r^N,$ so that $\UNs \simeq E^\sigma$.  Another special case occurs when $s<e$, so that $\UNs \simeq E$.

In the following result, we identify the representation $\UNs$ when $G=(2a,2,2)$ and $N=G(a,d,2)$ as in \Cref{thm:classification}($2$.c). In this case it is not possible to write $\UNs$ as a fake tensor power (\Cref{def:fake_tensor}) of $\VN$ in general.

\begin{proposition}\label{prop:infinite_normal_3} Let $a=de$ and $N=G(a,d,2)\triangleleft G(2a,2,2)=G$, and fix $\sigma:\zeta_{2a}\mapsto\zeta_{2a}^s$ for $1\leq s<2a$ with $\mathrm{gcd}(s,2a)=1$. Then $H\simeq C_2\times C_d$ and $\VN\simeq \lambda_2\oplus\lambda_d$, where $\lambda_2$ is the standard reflection representation of $C_2$ and $\lambda_d$ is the standard reflection representation of $C_d$, and $\UNs\simeq \lambda_2^{\otimes\left\lceil\frac{s}{a}\right\rceil}\oplus \lambda_d^{\otimes \left\lceil\frac{s}{e}\right\rceil}$ as $H$-modules.
\end{proposition}

\begin{proof} We may choose the fundamental $N$-invariants \[\n_1(\mathbf{x})=x_1^a+x_2^a \qquad\text{and}\qquad \n_2(\mathbf{x})=(x_1x_2)^e\] as in \Cref{eq:infinite_invariants}. Letting \[h_1:=\begin{pmatrix}0&\zeta_{2a}\\\zeta_{2a}^{-1} & 0\end{pmatrix}\qquad\text{and}\qquad h_2:=\begin{pmatrix}\zeta_a&0\\0&1\end{pmatrix}.\] We see that $h_i(\n_j)=\n_j$ if $i\neq j$, $h_1(\n_1)=-\n_1$, and $h_2(\n_2)=\zeta_d\n_2$, and therefore $h_1$ and $h_2$ map to the generators of the cyclic factors in the bicyclic quotient group $H\simeq C_2\times C_d$. By \Cref{thm:infinite_exponents,lem:fake_OS_space}, $\tilde{U}^N_\sigma=\spn\{\tilde{u}_1^N,\tilde{u}_2^N\}$ is isomorphic to $\UNsd$ as a graded $G$-module, where \begin{align*}\tilde{u}_1^N&=x_1^{\left\lceil\frac{s}{a}\right\rceil a -s}\otimes x_1^\sigma + x_2^{\left\lceil\frac{s}{a}\right\rceil a -s}\otimes x_2^\sigma;\qquad \text{and}\\ \tilde{u}_2^N&=(x_1x_2)^{\left\lceil \frac{s}{e}\right\rceil e}x_1^{-s}\otimes x_1^\sigma + (x_1x_2)^{\left\lceil \frac{s}{e}\right\rceil e}x_2^{-s}\otimes x_2^\sigma. \end{align*} We see that $h_i(\tilde{u}^N_j)=\tilde{u}^N_j$ if $i\neq j$, $h_1\tilde{u}_1^N=(-1)^{\left\lceil\frac{s}{a}\right\rceil}\tilde{u}_1^N$, and $h_2(\tilde{u}^N_2)=\zeta_d^{\left\lceil\frac{s}{e}\right\rceil}\tilde{u}_2^N$. \end{proof}

As before, note that we do not necessarily have that $\left\lceil\frac{s}{e}\right\rceil$ is relatively prime to $d$. When $d=a$ so that $e=1$, we have that $\UNs\simeq\VNs$ whenever $s<a$. When $d=1$ so that $e=a$, we do have $\UNs\simeq\VN^{\boxtimes\left\lceil\frac{s}{a}\right\rceil}$ as in the previous cases. Another special case occurs when $s<e$, so that $\UNs\simeq\VN$.

\subsection{\texorpdfstring{An exceptional example $G_{15}\label{sec:exceptional_example} \triangleright G_{12}$}{An exceptional example G12 in G15}}
To illustrate the choice of indexing of degrees and exponents that is needed in~\Cref{thm:numbers}, consider \[G:= G_{15}=\left\langle \left(\begin{array}{cc}1 & 0\\0 & -1\end{array}\right),\frac{\zeta_3}{2}\left(\begin{array}{cc}-1-i & 1-i\\-1-i & -1+i\end{array}\right),\frac{1}{\sqrt{2}}\left(\begin{array}{cc}1 & -1\\-1 & -1\end{array}\right) \right\rangle,\] where we use the conventions in~\cite[Chapter 3]{lehrer2009unitary} and the given matrices describe the action of $G$ on $V^*$. We will write the reflection generators of this non-well-generated reflection group $G$ as $\mathbf{s}$, $\mathbf{t}$, and $\mathbf{u}$, in the order they appear above. The degrees of $G$ are $d_1^G=12$ and $d_2^G=24$, with corresponding invariant generators of $S(V^*)^G$ given by \[G_1(x,y)=(x^5y - xy^5)^2 \qquad\text{and}\qquad G_2(x,y)=(x^8 + 14x^4y^4 + y^8)^3.\]

We then have the following normal reflection subgroup $N\trianglelefteq G$ generated by the reflections $\mathbf{u}$, $\mathbf{sus}^{-1}$, and $\mathbf{tsus}^{-1}\mathbf{t}^{-1}$:
\[N:= G_{12}=\left\langle\frac{1}{\sqrt{2}}\left(\begin{array}{cc}1 & -1\\-1 & -1\end{array}\right),\frac{1}{\sqrt{2}}\left(\begin{array}{cc}1 & 1\\1 & -1\end{array}\right),\left(\begin{array}{cc}0 & \zeta_8\\\zeta_8^{-1} & 0\end{array}\right) \right\rangle,\]
where again the above matrices describe the acion of $N$ on its reflection representation $V^{*}$. The degrees of $N$ are $d_1^N=6$ and $d_2^N=8$, with corresponding invariant generators of $S(V^*)^N$ given by \[N_1(x,y)=x^5y - xy^5 \qquad\text{and}\qquad N_2(x,y)=x^8 + 14x^4y^4 + y^8.\]

Then the quotient group $H:=G/N\simeq C_2 \times C_3$, and by~\Cref{thm:reflection_action} $H$ acts in a reflection representation on $\VN$, where $\VNd=\spn_\CC\{N_1,N_2\}$, since 
\begin{alignat*}{2}
\mathbf{s}(N_1)&=-N_1, \hspace{1em} &  \mathbf{s}(N_2)&=N_2\\
\mathbf{t}(N_1)&=N_1,\hspace{1em} &  \mathbf{t}(N_2)&=\zeta_3^2 N_2\\
\mathbf{u}(N_1)&=N_1, \hspace{1em}&  \mathbf{u}(N_2)&=N_2.
\end{alignat*}
Compatibly with the choices of invariant generators $\g_1$ and $\g_2$ for $G$ and $\n_1$ and $\n_2$ for $N$, we find invariant generators of $S(\VNd)^H$ that are $\mathbf{N}$-homogeneous of degrees $d_1^H=2$ and $d_2^H=3$:\[H_1(\n_1,\n_2)=N_1^2=\g_1(x,y) \qquad\text{and}\qquad H_2(\n_1,\n_2) = N_2^3=\g_2(x,y),\] and also $\mathbf{x}$-homogeneous of degrees $d_1^G=12=d_1^N\cdot d_1^H$ and $d_2^G=24=d_2^N\cdot d_2^H$.

As we can see from the explicit matrices above, the reflection representation of $G$ is actually defined over $\zeta_{24}$. For $s$ coprime to $24$ we write $\sigma_s$ for the Galois automorphism $\sigma_s\in\mathrm{Gal}(\mathbb{Q}(\zeta_{24})/\mathbb{Q})$ defined by $\sigma_s:\zeta_{24}\mapsto\zeta_{24}^s$. The complete data for every Galois twist $\Vs$ for $G=G_{15}$ and $N=G_{12}$ appears in~\Cref{table:example}.  Code for computing similar examples using Sage can be found at~\cite{nathancode}.

\begin{table}[htbp]
\[\begin{array}{|c|ccc|ccc|c|} \hline
s & \multicolumn{1}{|c@{\hspace*{\tabcolsep}\makebox[0pt]{$\cdot$}}}{d_i^N\!\!\!\!} & \multicolumn{1}{c@{\hspace*{\tabcolsep}\makebox[0pt]{$=$}}}{\! e_i^H(\VNs)} & e_i^G(\VNs)\! & \multicolumn{1}{|c@{\hspace*{\tabcolsep}\makebox[0pt]{$+$}}}{\!e_i^N(\Vs)} & \multicolumn{1}{c@{\hspace*{\tabcolsep}\makebox[0pt]{$=$}}}{e_i^G(\UNs)} & e_i^G(\Vs)\! &\!\! \sum\limits_{g \in G}\!\!\left( \prod\limits_{\lambda_i(g) \neq 1}\!\!\! \frac{1-\lambda_i(g)^\sigma}{1-\lambda_i(g)}\!\!\right) \!q^{\fix_V (g)} t^{\fix_{E} (g)}\!\! \\ \hline
1 & 6,8 & 1,2 & 6,16 & 5,7 & 6,16 & 11,23 & (qt+5t+6)(qt+7t+16)\\ 
5 & 6,8 & 1,1 & 6,8 & 1,11 & 6,8 & 7,19 & (qt+t+6)(qt+11t+8)\\
7 & 6,8 & 1,2 & 6,16 & 11,1 & 6,16 & 17,17 & (qt+11t+6)(qt+t+16)\\
11 & 6,8 & 1,1 & 6,8 & 7,5 & 6, 8 & 13,13 & (qt+7t+6)(qt+5t+8)\\
13 & 6,8 & 1,2 & 6,16 & 11,1 & 0,22 & 11,23 & (qt+11t)(qt+t+22)\\ 
17 & 6,8 & 1,1 & 6,8 & 7,5 & 0,14 & 7,19 & (qt+7t)(qt+5t+14)\\
19 & 6,8 & 1,2 & 6,16 & 5,7 & 0,22 & 5,29 &(qt+5t)(qt+7t+22)\\ 
23 & 6,8 & 1,1 & 6,8 & 1,11 & 0,14 & 1,25 & (qt+t)(qt+11t+14)\\ \hline
\end{array}\]

\caption{Data for $G=G_{15}$, $N=G_{12}$, and $H=G/N=C_2\times C_3$ computed using~\cite{nathancode}.  The rows are indexed by the Galois twists $\sigma_s:\zeta_{24}\to\zeta_{24}^s$ (for $s$ coprime to $24$). The columns contain the degrees of $N$ multiplied by the $\VNs$-exponents of $H$ to obtain the $\VNs$-exponents of $G$, and the $\Vs$-exponents of $N$ added to the $\UNs$-exponents of $G$ to obtain the $\Vs$-exponents of $G$, indexed according to~\Cref{thm:numbers}. The final column lists the corresponding factorization of the weighted sum over $G$.}
\label{table:example}
\end{table}

Let us work out the case $\sigma=\sigma_{13}$ in detail; the other cases are similar. Writing $x^\sigma$ and $y^\sigma$ for the basis vectors of $\Vsd$ as before, we find bases $\{u_1^G,u_2^G\}$ for $\UGsd$ and $\{u_1^N,u_2^N\}$ for $\UNsd$ given~by:
\begin{alignat*}{2}
u_1^G&=(x^{11} - 22 x^7 y^4 - 11 x^3 y^8) \otimes x^\sigma \\
&+(y^{11} - 22 x^4 y^7 - 11 x^8 y^3) \otimes y^\sigma ;\\
u_2^G&=\left((x^{21}y^2- xy^{22}) + 27(x^{17}y^6-x^5y^{18}) + 170(x^{13}y^{10} - x^9 y^{14})  \right)\otimes x^\sigma \\
&+\left((x^2y^{21}- x^{22}y) + 27(x^6y^{17}- x^{18}y^5) +170(x^{10}y^{13}- x^{14}y^9)\right)\otimes y^\sigma;\\
u_1^N&=u_1^G; \text{ and }\\
u_2^N&=y\otimes x^\sigma - x\otimes y^\sigma.
\end{alignat*}
Writing $u_i^G=a_{i1}^G\otimes X^\sigma+a_{i2}^G\otimes y^\sigma$ and $u_i^N=a_{i1}^N\otimes x^\sigma+a_{i2}^N\otimes y_\sigma$ as above, let us verify that $a_{ij}^G\in\CG$ and $a_{ij}^N\in\CN$. It is clear that $a_{11}^G,a_{12}^G\in\CG$, because $\mathrm{deg}(a_{ij}^G)=1<d_1^G,d_2^G$, and every polynomial of degree smaller than every degree of $G$ belongs to $\CG$. Similarly, we see that $a_{21}^N,a_{22}^N\in\CN$, since $\mathrm{deg}(a_{2j}^N)=1<d_1^N,d_2^N$. We observe that $u_2^G=\n_1\n_2^2\cdot u_2^N$ (we discuss the meaning of this observation in more detail below), and therefore $u_2^G\in\CG\otimes\Vsd\simeq\CH\otimes\CN\otimes\Vsd$ by \Cref{prop:coinvariant_decomposition}. To see that $u_1^N\in\CN\otimes \Vsd$ also, suppose that $w_1^N,u_2^N$ is a homogeneous basis for $\UNsd$, and let $p_1,p_2\in S(V^*)^N$ be homogeneous such that $u_1^N=p_1w_1^N+p_2u_2^N$. But then $p_2=0$, since no homogeneous $N$-invariant polynomial has degree $10$, which implies that $a_{1j}^N$ is divisible by a homogeneous $N$-invariant polynomial $p_1$ with $\mathrm{deg}(p_1)\leq 11$. But the only non-constant choices are $p_1=\n_1$ or $p_1=\n_2$, none of which divide the coefficients $a_{1j}^N$ above, and therefore $p_1\in\CC^\times$. This concludes the proof that the $u_i^G$ and $u_i^N$ specified above are indeed bases for $\UGsd$ and $\UNsd$, respectively, as claimed.

Hence, the $\Vs$-exponents of $G$ are $e_1^G(V^\sigma)=11$ and $e_2^G(V^\sigma)=23$, and the $\Vs$-exponents of $N$ are $e_1^N(V^\sigma)=11$ and $e_2^N(V^\sigma)=1$. We compute the action of $G$ on $\UNsd$ and obtain
\begin{align*}
g(u_1^N)&=u_1^N \text{ for all } g \in G,\\
\mathbf{s}(u_2^N)&=-u_2^N, \\
\mathbf{t}(u_2^N)&=\zeta_3^2 u_2^N, \text{ and }\\
\mathbf{u}(u_2^N)&=u_2^N.
\end{align*}

The resulting $G$-module structure on $\UNs$ yields the direct sum of the trivial representation with the unique 1-dimensional representation of $G$ with fake degree $q^{22}$, which is $G$-isomorphic to the $\mathbb{C}$-span of the $G$-harmonic semi-invariant homogeneous polynomial \[a_{22}^H:=(x^{21}y- xy^{21}) + 27(x^{17}y^5-x^5y^{17}) + 170(x^{13}y^9 - x^9y^{13})\in\CG\subset S(V^*),\] so that a basis for $(\CG\otimes\UNsd)^G$ is given by $\{1\otimes u_1^N,a_{22}^H\otimes u_2^N\}$.

On the other hand, the $H$-module structure on $\UNs$ yields the direct sum of the trivial representation with the $1$-dimensional representation of $H$ with fake degree $q^3$ given by the $\mathbb{C}$-span of the $H$-harmonic semi-invariant $\mathbf{N}$-homogeneous polynomial \[a_{22}^H=\n_1\n_2^2\in\CH\subset S(\VNd),\] so that a basis for $(\CH\otimes\UNsd)^H$ is given by $\{1\otimes u_1^N,\n_1\n_2^2\otimes u_2^N\}$.

Thus we witness the general isomorphism \[(\CG\otimes\Vsd)^G\simeq(\CH\otimes\UNsd)^H\] from the proof of \Cref{cor:general_numerology} in this example, since we obtain $a_{22}^H\otimes u_2^N\mapsto u_2^G$ by collapsing the first tensor in
\begin{align*}
a_{22}^H\otimes u_2^N&=(x^5y-xy^5)(x^8+14x^4y^4+y^8)^2\otimes(y\otimes x^\sigma-x\otimes y^\sigma)\in(\CH\otimes\UNsd)^H\\
&\mapsto (x^5y-xy^5)(x^8+14x^4y^4+y^8)^2(y\otimes x^\sigma-x\otimes y^\sigma)=u_2^G\in(\CG\otimes\Vsd)^G.
\end{align*}

\section{Acknowledgements}
We thank Theo Douvropoulos for many helpful remarks and suggestions.

\providecommand{\bysame}{\leavevmode\hbox to3em{\hrulefill}\thinspace}
\providecommand{\MR}{\relax\ifhmode\unskip\space\fi MR }
\providecommand{\MRhref}[2]{%
  \href{http://www.ams.org/mathscinet-getitem?mr=#1}{#2}
}
\providecommand{\href}[2]{#2}

\end{document}